\documentclass[a4paper,11pt,leqno]{amsart}
\usepackage[applemac]{inputenc}
\usepackage{epsfig,graphics}
\usepackage[all,knot]{xy}
\usepackage{amsthm,amssymb,amsbsy,bbm,a4wide,verbatim,url}
\usepackage{longtable}
\usepackage{rotating}
\xyoption{knot}

\def\neqnear{\mbox{\ \ $\backslash\!\!\!\!\!\!\nearrow$}}
\def\kk{\mathbbm{k}}
\def\Ad{\mathrm{Ad}}
\def\ad{\mathrm{ad}}
\def\Z{\mathbbm{Z}}
\def\C{\mathbbm{C}}
\def\Q{\mathbbm{Q}}
\def\C{\mathbbm{C}}
\def\R{\mathbbm{R}}
\def\un{\mathbbm{1}}
\def\tr{\mathrm{tr}\,}
\def\dd{\mathrm{d}\,}
\def\End{\mathrm{End}}
\def\Lie{\mathrm{Lie}\,}
\def\Ker{\mathrm{Ker}\,}
\def\Gal{\mathrm{Gal}}
\def\Aut{\mathrm{Aut}}

\def\GL{\mathrm{GL}}
\def\rk{\mathrm{rk}\,}
\def\gr{\mathrm{gr.}\,}
\def\Id{\mathrm{Id}}
\def\AA{\mathcal{A}}
\def\Imm{\mathrm{Im}\,}
\def\Hom{\mathrm{Hom}}
\def\Res{\mathrm{Res}}
\def\Ind{\mathrm{Ind}}
\def\g{\mathfrak{g}}
\def\h{\mathfrak{h}}
\def\OO{\mathrm{O}}
\def\ii{\mathrm{i}}
\def\la{\lambda}
\def\om{\omega}
\def\gl{\mathfrak{gl}}
\def\sl{\mathfrak{sl}}
\def\so{\mathfrak{so}}
\def\sp{\mathfrak{sp}}
\def\osp{\mathfrak{osp}}

\def\eps{\varepsilon}
\def\into{\hookrightarrow}
\def\onto{\twoheadrightarrow}
\def\ss{\mathfrak{s}}
\def\Irr{\mathrm{Irr}}

\newtheorem{theor}[equation]{Theorem}
\newtheorem{lemma}[equation]{Lemma}
\newtheorem{defi}[equation]{Definition}
\newtheorem{prop}[equation]{Proposition}
\newtheorem{cor}[equation]{Corollary}
\newtheorem{conj}[equation]{Conjecture}
\newtheorem{remark}[equation]{Remark}
\newtheorem{example}[equation]{Example}
\numberwithin{equation}{section}
\renewcommand{\labelenumi}{(\roman{enumi})}

\def\HGR{\mathcal{H}_{\mathrm{gr}}}
\def\HG2{\mathcal{H}_{\mathrm{tail}}}

\title{Infinitesimal Hecke Algebras IV}
\author{Ivan Marin}
\address{LAMFA, Universit\'e de Picardie Jules Verne,  Amiens, France}
\email{ivan.marin@u-picardie.fr}
\date{December 5, 2012}
\begin{document}

\maketitle

\begin{abstract}
We refine the infinitesimal Hecke algebra associated to a 2-reflection group into a $\Z/2\Z$-graded
Lie algebra, as a first step towards a global understanding of a natural $\mathbbm{N}$-graded object.
We provide an interpretation of this Lie algebra in terms of the image of the
braid group into the Hecke algebra, in connection with unitary structures. This connection is particularly
strong when $W$ is a Coxeter group, and when in addition we can use generalizations
of Drinfeld's rational even associators. Finally, in the Coxeter ADE case, we provide a generating
set of the even part of this Lie algebra, which originates from the rotation subgroup of the Coxeter group.
\end{abstract}

\tableofcontents

\section{Introduction}

This paper is a continuation of \cite{IH1,IH2,IH3}.  Let $W$ be a complex reflection group, that is a subgroup of $\GL_n(\C)$
generated by reflections (of order 2), and $\kk$ a field of characteristic $0$. We denote $\mathcal{R}$ the set of reflections
in $W$. In \cite{IH1,IH2} we defined the infinitesimal Hecke algebra $\mathcal{H}$ as the Lie sub-algebra of the group algebra $\kk W$
generated by $\mathcal{R}$. We proved there that it can be identified with the Lie algebra of the Zariski closure
of the braid group of $W$ inside the Hecke algebra (with generic parameter).
Here we first define a graded version of this Lie algebra. 
Let us introduce the group algebra $\kk[[h]]W$ of $W$ over the ring of formal series $\kk[[h]]$, and consider it
as a $\Z_{\geq 0}$-graded $\kk$-algebra.

\begin{defi} The graded infinitesimal Hecke algebra $\HGR$ is the Lie subalgebra
of $\kk[[h]] W$ generated by the $hs, s \in \mathcal{R}$.
\end{defi}

An approximation of this still mysterious graded Lie algebra is given by $\HGR \otimes_{\kk[[h]]} \kk((h)) = \mathcal{H} \otimes_{\kk} \kk((h))$,
whose structure is known from \cite{IH2}. A better approximation is given by the following object, which is the central focus of
this paper.

\begin{defi} The `tail' infinitesimal Hecke algebra $\HG2$ is the Lie $\kk(h^2)$-subalgebra of $\kk(h) W$ generated by
 the $hs, s \in \mathcal{R}$.
\end{defi}

Note that
$\HG2$ has a natural structure of $\Z/2$-graded Lie algebra over $\kk(h^2)$.
Let $\HGR^{(r)}$ denote
the homogeneous part of degree $r$ of $\HGR$. It is a $\kk$-vector space which can be clearly identified to a subspace of $\mathcal{H}$, that is $\HGR^{(r)} = h^r \HGR^r$ with $\HGR^r \subset \mathcal{H}$.

\begin{prop} (see proposition \ref{propstructgrbase}) As subspaces of $\mathcal{H}$, for all $r \geq 1$, $\HGR^r \subset \HGR^{r+2}$. 
\end{prop}

Let $\HGR^{2 \infty}$ denote the union of the $\HGR^{2r}$, and $\HGR^{2 \infty + 1}$ denote the union of the
$\HGR^{2r+1}$ for $r \geq 0$. We let $\HG2^{0}$ and $\HG2^{1}$ denote the homogeneous components of $\HG2$.
We adopt the following convention : whenever $V$ is a $\kk$-subspace of the group algebra $\kk((h)) W$,
and $R$ is a $\kk$-subalgebra of $\kk((h))$, we let $R V$ denote the image in $\kk((h)) W$ of the
natural morphism $V \otimes_{\kk} R \to \kk((h)) W$. The following is a straightforward consequence of
proposition \ref{propstructgrbase}.

\begin{prop} $\HG2^0  =  \kk(h^2) \HGR^{2 \infty}$ and $\HG2^1=  h \kk(h^2) \left( Z(\mathcal{H}) \oplus  \HGR^{2 \infty + 1} \right)$.

\end{prop}

This means that $\HG2$ indeed captures the 'tail' of the sequence $(\HGR^r)_{r \geq 0}$. The Lie algebra
$\HG2$ has a meaning in its own right, that we explain now. 
Recall from \cite{IH2,IH3}  the existence of morphisms $\Psi_H : H(q) \to KW$, a priori
only for $\kk = \C$ but also conjecturally for $\kk = \Q$, where $H(q)$  is the Hecke algebra associated to $W$ ; 
more precisely, $H(q)$ is defined over $K = \kk((h))$,
$q = e^h$, and $H(q)$ is the quotient of the group algebra $KB$  of the braid group $B$ associated to $W$ by relations $(\ss -q)(\ss+q^{-1}) = 0$, where $\ss$ runs
among the braided reflections of $B$.

When $W$ is a Coxeter group, it is known that these morphisms provide isomorphisms $H(q) \simeq K W$. In the general case, it is also conjectured
to be the case. We call this statement the weak BMR conjecture, and refer to \cite{CYCLO} for a detailed discussion of it.
It is known to hold true for the general series $G(de,e,n)$ of irreducible complex reflection groups as
well as for all the groups of rank 2, plus some other cases.

We let $\sigma \in \Aut(K)$ be defined by $f(h) \mapsto f(-h)$, and we let $K^{\sigma} = \kk((h^2))$.  
We prove the following theorem

\begin{theor} \label{theoZarHecke} Under the weak BMR conjecture for $W$, the Zariski closure of $\Psi_H(B)$ inside $H(q)^{\times}$ considered as
an algebraic group over $\kk((h^2))$ has for Lie algebra $\HG2\otimes_{\kk(h^2)} \kk((h^2))$.
\end{theor}

Recall from \cite{IH1}, \cite{IH2} that $\mathcal{H} \subset \mathcal{L}_{\eps}(W)$, where $\mathcal{L}_{\eps}(W)$ is
the Lie subalgebra of $\kk W$ spanned by the $g - \eps(g) g^{-1}$ for $g \in W$, where $\eps : W \to \{ \pm 1 \}$
is the sign morphism, induced by the determinant. We denote $\OO= \Ker \eps$ the rotation subgroup of $W$. In section \ref{sectlepsG} we endow $\mathcal{L}_{\eps}(W)$
with the structure of a $\Z/2$-graded Lie algebra $\mathcal{L}_{\eps}(W) = \mathcal{L}_1(\OO) \oplus \mathcal{L}_1(\OO^{\dagger})$,
such that
$$
\HG2^0 = \HG2 \cap \kk(h^2) \mathcal{L}_1(\OO) \mbox{ \ and \ } h^{-1} \HG2^1 = \HG2 \cap \kk(h^2) \mathcal{L}_1(\OO^{\dagger})
$$

In the same section \ref{sectlepsG} we determine the structure as $\Z/2$-graded Lie algebras of $\mathcal{L}_{\eps}(W)$ and of the $\rho(\mathcal{L}_{\eps}(W))$, for $\rho$
an irreducible representation of $W$, notably when $W$ is a Coxeter group.

We then prove the following theorem.

\begin{theor}\label{theo16} (see propositions \ref{propZarLieP} and \ref{propimageorthL1})
Let $W$ be a reflection group, $\rho : W \to \gl_N(\kk)$ an irreducible representation, and $\Psi : B \to \GL_N(K)$
the representation of $H(q)$ associated to $\rho$. The closure of $\Psi(B)$ in the $K^{\sigma}$-group
$\GL_N(K)$ has for Lie algebra $K^{\sigma} \rho(\HG2^0) \oplus h K^{\sigma} \rho(\HG2^1)$. If $\rho^* \simeq \rho$
(e.g. if $W$ is a Coxeter group) then 
$$
\begin{array}{lcl}
K^{\sigma} \rho(\HG2^0) &=& K^{\sigma} \left( \rho(\mathcal{H}) \cap \rho(\mathcal{L}_1(\OO)) \right) \\
K^{\sigma} \rho(\HG2^1) &=& h K^{\sigma} \left( \rho(\mathcal{H}) \cap \rho(\mathcal{L}_1(\OO^{\dagger})) \right) \\
\end{array}
$$
\end{theor}

In the latter case, we moreover prove that this closure is the intersection of the closure over $K$
with a suitable unitary group (see proposition \ref{propintersectunitary}).

A natural question is whether $\HG2^0 \subset \kk(h^2) \OO$ can be defined
intrinsically from $\OO$, or needs instead to be defined inside $\kk(h^2 ) W$. 
Let $\mathcal{A}$ denote the Lie subalgebra of $\kk \OO$
generated by elements $[s,u] = su-us = su - (su)^{-1}$ for $s,u \in \mathcal{R}$. In sections \ref{sectrot}, \ref{sectrotA} and \ref{sectrotDE}, 
we prove the following.

\begin{theor}  $\kk(h^2)\mathcal{A}$ is contained in $\HG2^{0}$, and $\kk(h^2)\mathcal{A} = \HG2^{0}$ when
$W$ is a Coxeter group of type ADE.
\end{theor}

As opposed to the non-graded setting, where all the results could be stated and prove (sometimes
to the expense of admitting a couple of natural or well-established conjectures) in the most
general setting of complex pseudo-reflection groups (see \cite{IH3}) it seems that this $\Z/2\Z$-graded
version behaves more naturally in the case of Coxeter groups. In \S 6 we detail some special aspects
of Coxeter groups that are relevant here, and state a refinement of the conjectures of \cite{IH2,IH3}
which is specific to the Coxeter setting. Loosely speaking, this refinement is connected to
the existence of analogues of Drinfeld's \emph{even} associators. Under this conjecture,
which imply special cases of a classical result of Lusztig and is shown to hold in
Coxeter types $A,B_n/C_n$ and $I_2(m)$, $\HG2^0$ can be identified with the `orthogonal part' of
the image inside the Hecke algebra of the pro-unipotent completion of the pure braid group. Ironically,
one can prove that the image of the pure braid group itself has no `orthogonal part' (see remark \ref{remEXPLAIN}).

The next natural goal would be to understand the full graded version $\HGR$. Another purpose
is to use this broader knowledge provided by the $\Z/2\Z$-graded version in the
case of specialized values of the parameter $q$. We already showed in \cite{TLFP}, in the case
$W = \mathfrak{S}_n$ and over a finite field, that the image of the braid group may fill in
either full linear groups, or instead preserve a unitary form, depending on whether the
unitary form existing at the generic level admits an avatar in the finite field and for the
parameter $q$ under consideration. 

In order to check some representation-theoretic facts on specific complex reflection groups,
we made use of the development version of the CHEVIE package for GAP3. This software
is maintained by Jean Michel  and can be found at 
\url{http://www.math.jussieu.fr:/~jmichel/chevie/index.html}. It will simply be referred to as `CHEVIE' in the sequel.

\medskip

{\bf Acknowledgements.} I thank the University of Tokyo and T. Kohno for a one-month stay in 2009 where
I started this work,
 M. Cabanes for help in proving lemma \ref{clifsubtletypeD},
J. Michel and C. Bonnaf\'e for the reference \cite{CARTER}, C. Cornut for useful discussions about unitary groups. For the tables below, and also for the determination of $\mathcal{A}$ in small rank, this work made heavy use of a fast computer
funded by the ANR grant ANR-09-JCJC-0102-01.

\section{Infinitesimal graded Hecke algebras}

Since $\HGR$ is generated as a Lie algebra by a conjugacy class
in $W$, each of the homogeneous components $\HGR^r$ is
a $W$-submodule of $\mathcal{H}$. We first prove the
following
\begin{prop} \label{propstructgrbase}
\begin{enumerate}
\item $\HGR^1 = Z(\mathcal{H}) \oplus \left( \mathcal{H}' \cap \HGR^1\right)$ and $\mathcal{H}' \cap \HGR^1 \subset \HGR^3$
\item For all $r \geq 2$, $\HGR^r \subset \mathcal{H}' = [\mathcal{H},\mathcal{H}]$ and $\HGR^r \subset \HGR^{r+2}$. 
\end{enumerate}
\end{prop} 
\begin{proof}
(i) is clear from the general properties of $\mathcal{H}$. We prove (ii) and assume $r \geq 2$. Again clearly
$\HGR^r \subset \mathcal{H}' = [\mathcal{H},\mathcal{H}]$.
Let $x \in \HGR^r$ and $s \in \mathcal{R}$. One has $[s,[s,x]] = 2(sxs-x) \in \HGR^{r+2}$,
hence $s.x - x \in \HGR^{r+2}$ where $g .x$ for $g \in W$ and $x \in \mathcal{H}$
denotes the conjugation action.
We introduce $M = \sum_{s \in \mathcal{R}} (s-1).\HGR^r$. By the above $M \subset \HGR^{r+2}$,
and also $M \subset \HGR^r$ because $\HGR^r$ is $W$-stable.
We prove that, when $r \geq 2$, then $M = \HGR^r$. For this,
to each $W$-submodule $U$ of $\HGR^r$ we associate
the subspace $M_U = \sum_{s \in \mathcal{R}} (s-1).U$. Since $U$ is a
$W$-submodule we get $M_U \subset U$,
and moreover $M_U$ is stabilized by $W$ : if $w \in W$, then
$$
w.\left(\sum_{s \in \mathcal{R}} (s-1).U\right)
= \sum_{s \in \mathcal{R}} (ws-w).U
= \sum_{s \in \mathcal{R}} (wsw^{-1}-1).(w.U)
= \sum_{s \in \mathcal{R}} (s-1).U = M_U.
$$
When $U \neq 0$, $M_U = 0$ if and only if $\forall s \in \mathcal{R} \forall x \in U \ s.x = x$,
that is $\forall w \in W \forall x \in U \ \ w.x = x$, hence $\un \into
U$. But when $r \geq 2$, $\HGR^r \subset \mathcal{H}' \subset (\kk W)'$, which as
a $W$-module (under the conjugation action) is a linear complement of $Z(W)$ in $\kk W$.
From this we deduce, when $r \geq 2$, that $U \neq 0 \Rightarrow M_U \neq 0$,
hence if  $U$ is irreductible then $M_U = U$. Since $\HGR^r$ is a sum of 
irreducible submodules this proves $\HGR^r = M \subset \HGR^{r+2}$.
\end{proof}

For a given group $W$, the computation of the grading can be
made as follows. Let $\mathcal{M}_r \subset \mathcal{H}$
denote the subspace spanned by the brackets of at most $r$
reflections, that is $\mathcal{M}_1 = \kk \mathcal{R}$ and $\mathcal{M}_{r+1}
 = \mathcal{M}_r + [\mathcal{R},\mathcal{M}_r]$. The datas of $\dim \mathcal{M}_r$
 determines $\dim \HGR^r$ as follows. First of all, $\mathcal{M}_1 = \kk \mathcal{R}
  = \HGR^1$. Then $\mathcal{M}_2 = \mathcal{M}_1 + \HGR^2$
  hence $\dim \HGR^2 = \dim \mathcal{M}_2 - \dim \mathcal{M}_1$; 
  $\mathcal{M}_3 = \mathcal{M}_2 + \HGR^3 + \HGR^1 = 
  \HGR^2 + \HGR^3 +( \HGR^1 \cap \mathcal{H}') + Z(\mathcal{H})$,
 thus $\mathcal{M}_3 = Z(\mathcal{H})  \oplus \HGR^2 \oplus \HGR^3$
  and $\dim \HGR^3 =\dim \mathcal{M}_3  -\dim \HGR^2 - \dim Z(\mathcal{H})$. 
  
For $n \geq 4$, $\mathcal{M}_n = \mathcal{M}_{n-1} + \HGR^n$,
and $\mathcal{M}_{n-1} \cap \HGR^n = \HGR^{n-2}$
hence $\dim \mathcal{M}_n = \dim \mathcal{M}_{n-1} + \dim \HGR^n - \dim \HGR^{n-2}$ 
and $\dim \HGR^n = \dim \mathcal{M}_n - \dim \mathcal{M}_{n-1} + \dim \HGR^{n-2}$.
Summing up all these equalities yields to 
$$
\dim \HGR^n  =\dim \mathcal{M}_n -  \dim \HGR^{n-1}- \dim Z(\mathcal{H})
$$
whenever $n \geq 3$. A tabulation of the dimensions for some of the irreducible groups is given in
tables \ref{tableexcept} and \ref{tableclassical} (the figures in bold constitute the repeating pattern).

For the record, we recall the following result for Coxeter groups
(see \cite{CARTER}, lemma 2).
\begin{lemma} If $W < \GL_n(\R)$ is a real reflection group and $g \in W$
with $k = \rk(g-1)$, then $g$ is a product of $k$ reflections.
\end{lemma}
\begin{proof} The main point is to show that, if $w \in W$
has no (nonzero) fixed vector, then there is  $s \in \mathcal{R}$
such that $ws$ has a fixed vector. For this one notes that
$w-1$ is invertible. Choosing $s \in \mathcal{R}$ and $v \in \R^n\setminus \{ 0 \}$
with $s.v = -v$, let $\tilde{v}$ such that $(w-1).\tilde{v} = v$.
Then $w.\tilde{v} = \tilde{v}+ v$,
hence $\langle \tilde{v},\tilde{v} \rangle = \langle w.\tilde{v},w.\tilde{v} \rangle
 = \langle \tilde{v},\tilde{v}\rangle + 2 \langle v,\tilde{v} \rangle + \langle v,v \rangle$
 hence $2 \langle v, \tilde{v} \rangle / \langle v,v \rangle = -1$. This implies
 $s. \tilde{v} = \tilde{v} + v$, hence $ws$ has $v + \tilde{v} \neq 0$
 for fixed vector.
\end{proof}

 \begin{remark} This statement is not true for complex reflection groups.
 Actually, we checked that for \emph{every} non-Coxeter
irreducible 2-reflection group $W$ of exceptional type, the maximal length $m(W)$
of an element with respect to the generating set of all reflexions
is always greater than the rank of $W$.
More precisely, we get the following table
$$
\begin{array}{ccc|ccc|ccc|}
W & \rk W & m(W) & W & \rk W & m(W) &  W & \rk W & m(W) \\
\hline
G_{12}& 2 & 3 & G_{24} & 3 & 4 & G_{31} &4  & 6 \\
G_{13}& 2 & 3 & G_{27} & 3 & 5 & G_{33} & 5 & 7 \\
G_{22}& 2 & 3 & G_{29} & 4 & 6 & G_{34} & 6 & > 6 \\
\end{array}
$$

As far as the infinite series $G(e,e,r)$ and $G(2e,e,r)$ are concerned, there seems to be a similar phenomenon, as illustrated
by the tables below. This strongly suggests that this property is an elementary characterization of Coxeter groups among
irreducible 2-reflection groups.

 $$
\begin{array}{lc|lc|lc|lc|lc}
W & m(W) &W & m(W) & W & m(W)& W & m(W)&W & m(W) \\
\hline
G(3,3,3) & 4 & G(4,4,3) & 4 &  G(5,5,3) & 4 & G(6,6,3) & 4 & G(7,7,3) & 4 \\ 
G(3,3,4) & 5 &G(4,4,4) & 6 &  G(5,5,4) &  6 \\
 G(3,3,5) & 6 &G(4,4,5) & 7  \\
\end{array}
$$
{}
$$
\begin{array}{lc|lc|lc|lc|lc}
W & m(W)&W & m(W)& W & m(W)& W & m(W)&W & m(W)\\
\hline
G(4,2,2) &  3  & G(6,3,2) & 3 & G(8,4,2) & 3 &G(10,5,2) & 3 & G(12,6,2) & 3 \\
G(4,2,3) &  4  & G(6,3,3) & 5 & G(8,4,3) & 5 & G(10,5,3) & 5 & G(12,6,3) & 5 \\
G(4,2,4) &  6    & G(6,3,4) & 6 & G(8,4,4) & 7 &  G(10,5,4) & 7 & G(12,6,4) & 7\\
G(4,2,5) & 7 \\
\end{array}
$$
 
 \end{remark}

 \begin{example} (Dihedral groups)
 Let $W$ be a dihedral group of order $2m$,
 with generators $s,\om$, relations $s^2 = \om^m = 1$, $s\om s = \om^{-1}$. We have
$\mathcal{R} = \{ s \om^k ; 0 \leq k \leq m-1 \}$,  
 and $\HGR^1 = \kk \mathcal{R}$, has dimension $m$.
$\HGR^2 $ is spanned by the $[s,s\om^k]$, that is
by the $\om^k - \om^{-k}$ for $0 \leq k \leq m-1$,
hence $\dim \HGR^2 = \lfloor \frac{m-1}{2} \rfloor$ .
$\HGR^3$ is spanned by the $[s \om^i, \om^k - \om^{-k}]$,
that is (as $car. \kk \neq 2$) by the $s \om^{i+k} - s \om^{i-k}$,
thus $\dim \HGR^3 = m - \dim \Z(\mathcal{H})$,
as $\dim \Z( \mathcal{H}) = 1$ if $m$ is odd, $\dim \Z( \mathcal{H}) = 1$ if $m$ is even.
Since $\dim \mathcal{H}' = 3 \lfloor \frac{m-1}{2} \rfloor$, we have $\HGR^3 = \HGR^5 = \dots = \HGR^{2 \infty + 1}$
and $\HGR^2 = \HGR^4 = \dots = \HGR^{2 \infty}$. 
\end{example}

\begin{table}
$$
\begin{array}{l|cccccccccc}
\hline
 & \HGR^1 & \HGR^2 & \HGR^3 & \HGR^4  & \HGR^5 & \HGR^6 &     \HGR^7  & \HGR^8 & \HGR^9 & \HGR^{10} \\
\hline
G_{12} & 12 &  \mathbf{11} &  \mathbf{16} & 11&  16&  \dots \\
\hline
G_{13} & 18& \mathbf{23}& \mathbf{28} & 23& 28& 23 &\dots \\
\hline
G_{22} &  30& \mathbf{59}& \mathbf{69}& 59& 69& 59 &\dots\\
\hline
G_{23} = H_3 &  15 &  \mathbf{22}& \mathbf{33}& 22& 33& 22 &\dots\\
\hline
G_{24} &  21& 49& \mathbf{90}& \mathbf{72} & 90& 72& 90 &\dots \\
\hline
G_{27} & 45& 237& \mathbf{551}& \mathbf{512}& 551& 512& 551 &\dots\\
\hline
G_{28} = F_4 &  24& 50& 142& 196& \mathbf{280}& \mathbf{228} & 280& 228 &\dots\\
\hline
\end{array}
$$
\caption{Grading for some exceptional groups}
\label{tableexcept}
\end{table}

\begin{table}
$$
\begin{array}{l|cccccccccccccc}
\hline
 & \HGR^1 & \HGR^2 & \HGR^3 & \HGR^4  & \HGR^5 & \HGR^6 &     \HGR^7  & \HGR^8 & \HGR^9 & \HGR^{10} & \HGR^{11} & \HGR^{12}& \HGR^{13} & \HGR^{14} \\
\hline
\hline
\mathfrak{S}_3 &  3& \mathbf{1}& \mathbf{2}& 1& 2   & \dots \\
\hline
\mathfrak{S}_4 &  6& \mathbf{4}& \mathbf{7}& 4& 7  & \dots \\
\hline
\mathfrak{S}_5 &  10& 10& 19& \mathbf{16}& \mathbf{23}& 16& 23&  \dots \\
\hline
\mathfrak{S}_6 &  15& 20& 44& 56& 92& 92& 122& \mathbf{112} & \mathbf{136}& 112& 136   & \dots \\
\hline
\mathfrak{S}_7 &  21& 35& 90& 161& 342& 533& 838& 987& 1081& \mathbf{1002} & \mathbf{1087} & 1002& 1087 & \dots  \\
\hline
\hline
B_3 &  9& \mathbf{7}& \mathbf{12}& 7& 12 & \dots \\
B_4 &16& 22& 50& 53& 77& \mathbf{59}& \mathbf{80}& 59& 80  & \dots \\
B_5 &  25& 50& 153& 301& 591& 701& 842& \mathbf{761}& \mathbf{869}& 761& 869  & \dots \\
\hline
\hline
D_4 & 12& 16& 35& \mathbf{32}& \mathbf{46}& 32& 46   & \dots \\
D_5 & 20& 40& 119& 216& 372& 381& 445& \mathbf{391}& \mathbf{449}& 391& 449   & \dots \\
\hline
\end{array}
$$
\caption{Grading for classical Coxeter groups}
\label{tableclassical}
\end{table}

\section{Preliminaries on $\Z/2$-algebras }

Let $\kk$ denote a field of characteristic $0$. Recall that a $\Z/2$-graded Lie algebra $\g = \g_0 \oplus \g_1$,
also called a symmetric pair ($\g_0,\g_1)$, is the same thing as a Lie algebra $\g$ endowed with an involution
in $\Aut(\g)$. A trivial kind of $\Z/2$-graded Lie algebra is the \emph{double} of the
Lie algebra $\g$, that is $\g \oplus \g$.

\medskip

We describe a list of  $\Z/2$-graded Lie algebras, that we will call the classical $\Z/2$-graded Lie algebras. More precisely,
in the sequel we will say that a $\Z/2$-graded Lie algebra over
$\kk$ is classical of type (x) if \emph{up to some extension of scalars} it is isomorphic to the $\Z/2$-graded Lie algebra
described at item (x) below. The semisimple part of the cases (a)-(f) below belong to the types (AI), (BDI), (CI)
(DIII), (CII), (AII) (in this order) in the Cartan classification of symmetric pairs.

\renewcommand{\labelenumi}{(\alph{enumi})}
\begin{enumerate}
\item $\gl_N(\kk)$. Let $<\ ,\ >$ a non-degenerate symmetric bilinear form
on $\kk^N$, and $x \mapsto x^+$ the adjoint operation on $\End(\kk^N)$. Then
$\gl_N(\kk)^0 = \{ x \ | \ x^+ = -x \} = \so_N(\kk)$
and $\gl_N(\kk)^1 = \{x \ | \ x^+ = x \}$.
\item $\so_{2N}(\kk)$ w.r.t. a  non-degenerate symmetric bilinear form
$<\ ,\ >$. Let $U,V \subset \kk^N$ with $\kk^{2N} = U \oplus V$,
$\dim U = \dim V = N$ and $U \perp V$ (this implies that the restriction
of $< \ , \ >$ to $U$ and $V$ is non-degenerate). Then
$$
\so_{2N}(\kk)^0 = \left\lbrace \left. \left( \begin{array}{cc} x & 0 \\ 0 & y 
\end{array} \right) \right| x^+ = -x, y^+ = -y \right\rbrace \simeq
\so(U) \times \so(V)
$$
and
$$\so_{2N}(\kk)^1 = \left\lbrace \left. \left( \begin{array}{cc} 0 & -m^+ \\ m & 0 
\end{array} \right) \right| m \in \gl_N(\kk) \right\rbrace 
$$
(where $m \mapsto m^+$ here denotes the adjonction $\Hom(U,V) \to \Hom(V,U)$). 
\item $\sp_{2N}(\kk)$ w.r.t. a  non-degenerate skew-symmetric bilinear form
$<\ ,\ >$. Let $U,V \subset \kk^{2N}$, $\kk^{2N} = U \oplus V$,
$\dim U = \dim V$, and $U,V$ totally isotropic. Assuming $< \ , \ > = \left(
\begin{array}{cc} 0 & 1 \\ -1 & 0 \end{array} \right)$ we can identify $U$ and
$V$. Then
$$
\sp_{2N}(\kk)^0 = \left\lbrace \left. \left( \begin{array}{cc} m & 0 \\ 0 & -m^+ 
\end{array} \right) \right| m \in \gl(U) 
\right\rbrace \simeq \gl(U)
$$
and
$$
\sp_{2N}(\kk)^1 = \left\lbrace \left. \left( \begin{array}{cc} 0 & x \\ y & 0 
\end{array} \right) \right| x^+ = x, y^+ = y \right\rbrace 
$$
\item $\so_{2N}(\kk)$  w.r.t. a  non-degenerate symmetric bilinear form
$<\ ,\ >$, written as $\left( \begin{array}{cc} 0 & 1 \\ 1 & 0 
\end{array} \right)$. Let $U,V \subset \kk^{2N}$, $\kk^{2N} = U \oplus V$,
$\dim U = \dim V$, and $U,V$ totally isotropic. Identifying $U$ with $V$ using $<\ , \ >$,
then
$$
\so_{2N}(\kk)^0 = \left\lbrace \left. \left( \begin{array}{cc} m & 0 \\ 0 & -m^+ 
\end{array} \right) \right| m \in \gl_N(\kk) \right\rbrace \simeq \gl_N(\kk) 
$$ 
and
$$
\so_{2N}(\kk)^1 = \left\lbrace \left. \left( \begin{array}{cc} 0 & x \\ y & 0 
\end{array} \right) \right| x^+ = -x, y^+ = -y \right\rbrace 
$$
\item $\sp_{2N}(\kk)$ for $N$ even w.r.t. a  non-degenerate skew-symmetric bilinear form
$<\ ,\ >$, written as $\left( \begin{array}{cc} J & 0 \\ 0 & J 
\end{array} \right)$ with $J = \left( \begin{array}{cc} 0 & 1 \\ -1 & 0 
\end{array} \right)$. Then
$$
\sp_{2N}(\kk)^0 = \left\lbrace \left. \left( \begin{array}{cc} x & 0 \\ 0 & y 
\end{array} \right) \right| x,y \in \sp_N(\kk)\right\rbrace \simeq \sp_N(\kk) \times \sp_N(\kk)
$$
(remark : here $m^+ = J^{-1}\ ^t m J$) and
$$
\sp_{2N}(\kk)^1 = \left\lbrace \left. \left( \begin{array}{cc} 0 & -m^+ \\ m & 0 
\end{array} \right) \right|  m \in \gl_N(\kk) \right\rbrace 
$$
\item $\gl_N(\kk)$ with $N$ even. Analogous to case (a), with $< \ , \ >$ a non-degenerate
\emph{skew-symmetric} bilinear form on $\kk^N$. Then $\gl_N^0(\kk) \simeq \sp_N(\kk)$.
\end{enumerate}

\renewcommand{\labelenumi}{(\arabic{enumi})}

We future use, we note that the group 
$$U_N^{\sigma}(K) = \{ M \in \GL_N(K) \ | \ M \, ^t \sigma(M) = 1 \}$$
is a $K^{\sigma}$-subgroup of $\GL_N(K)$, made of the
$a+hb \in \GL_N(K)$ such that $a \ ^t b$ is a symmetric matrix and $\ ^t a a = 1 + x ^t b$. Its  Lie algebra with its natural
$(\Z/2\Z)$-grading, is made of the $\alpha + h \beta$ with $\ ^t \alpha = - \alpha$ and $\beta = \ ^t \beta$, hence is of type (a).

\section{The Lie algebra $\mathcal{L}_{\eps}(G)$}

\label{sectlepsG}
\subsection{General facts}

Let $G$ be a finite group and $\alpha : G \to \{ \pm 1 \}$ a character (for instance the trivial character $\alpha = 1$). 
Then Lie algebra $\mathcal{L}_{\alpha}(G)$
(see \cite{LIEALG}) is the subspace of $\kk G$ spanned by the $g - \alpha(g)g^{-1}$ for $g \in G$.
We assume there is a nontrivial character  $\eps : G \onto \{ \pm 1 \}$,
and let $A = \Ker \eps$, $A^{\dagger} = G \setminus A$. Then $\mathcal{L}_{\eps}(G)$ has a natural $\Z/2$-grading, given by
$\mathcal{L}_{\eps}(G)^0 = \mathcal{L}_1(A)$ and $\mathcal{L}_{\eps}(G)^1 = \mathcal{L}_1(A^{\dagger}) := < b + b^{-1} \ | \ b \in A^{\dagger}>$. 
Let 
$$
\mathcal{E}  = \{ \rho \in \Irr(G) \ | \  \rho^* \otimes \eps \not\simeq \rho \} \ \ 
\mathcal{F}^+ =  \{ \rho \in \Irr(G) \ | \ \eps \into S^2 \rho \} \ \ 
\mathcal{F}^- =  \{ \rho \in \Irr(G) \ | \ \eps \into \Lambda^2 \rho \} 
$$
{}
and $\mathcal{F} = \mathcal{F}^+ \cup \mathcal{F}^- = \{ \rho \in \Irr(G) \ | \ \eps \into \rho \otimes \rho \}$
(clearly $\mathcal{F}^+ \cap \mathcal{F}^- = \emptyset$). We denote $\sim$ the equivalence
relation
on $\Irr(G)$ generated by
$\rho^* \otimes \eps \sim \rho$ and we identify the set of equivalence
classes $\mathcal{E}/\sim$ with a system of representatives in $\mathcal{E}$. 

We assume that $\kk$ is a field of realizability for all the irreducible representations of $G$, and denote
$V_{\rho}$ the underlying $\kk$-vector space of $\rho$, that is $\rho : G \to \GL(V_{\rho})$. Then,
we have
$$
\kk G = \left( \bigoplus_{\rho \in \mathcal{E}/\sim} \left( \gl(V_{\rho}) \oplus \gl(V_{\rho^* \otimes \eps}) \right) \right)
\oplus \left( \bigoplus_{\rho \in \mathcal{F}} \gl(V_{\rho}) \right)
$$
and, according to \cite{LIEALG},
$$
\mathcal{L}_{\eps}(G) = 
 \left( \bigoplus_{\rho \in \mathcal{E}/\sim}  \gl(V_{\rho}) \right)
\oplus \left( \bigoplus_{\rho \in \mathcal{F}} \osp(V_{\rho}) \right)
$$
where the orthosymplectic Lie algebra $\osp(V_{\rho})$ is defined by the bilinear form induced by
$\eps \into \rho \otimes \rho$, and $\gl(V_{\rho}) \into  \gl(V_{\rho}) \oplus \gl(V_{\rho^* \otimes \eps})$
is given by $x \mapsto (x,-^t x)$.

\subsection{Decomposition of the components}

From now on we assume that all irreducible representations of $A$ are also realizable over $\kk$.
We will need a `real version' of Clifford theory. It is provided that the following
two lemmas, whose proof is an easy exercise in character theory which is left to the reader.

\begin{lemma} \label{lemclifreal} Let $G$ be a finite group, $\eps : G \onto \{ \pm 1 \}$, $A = \Ker \eps$,
$V \in \Irr(G)$ of real type with $V \simeq V \otimes \eps$. One has $\Res_A V  = V_+ \oplus V_-$.
Then
\begin{enumerate}
\item $V_+$ and $V_-$ are of the same type, which is either real or complex.
\item $\eps \into S^2 V$ iff $V_+$ and $V_-$ have real type.
\item $\eps \into \Lambda^2 V$ iff $V_+$ and $V_-$ have complex type.
\end{enumerate}
\end{lemma}

\begin{lemma} \label{lemclifquat} Let $G$ be a finite group, $\eps : G \onto \{ \pm 1 \}$, $A = \Ker \eps$,
$V \in \Irr(G)$ of quaternionic type with $V \simeq V \otimes \eps$. One has $\Res_A V  = V_+ \oplus V_-$.
Then
\begin{enumerate}
\item $V_+$ and $V_-$ are of the same type, which is either quaternionic or complex.
\item $\eps \into S^2 V$ iff $V_+$ and $V_-$ have complex type.
\item $\eps \into \Lambda^2 V$ iff $V_+$ and $V_-$ have quaternionic type.
\end{enumerate}
\end{lemma}

\begin{prop} \label{propL1Greal} If $\rho$ has real type, then
\begin{enumerate}
\item $\rho( \mathcal{L}_1(A)) \cap \rho( \mathcal{L}_1(A^{\dagger})) = 0$.
\item If $\rho \otimes \eps \not\simeq \rho$ then $\rho(\mathcal{L}_{\eps}(G))$ is a classical $\Z/2$-graded Lie algebra of type (a).
\item If $\eps \into S^2 \rho$ then $\rho(\mathcal{L}_{\eps}(G))$ is a classical $\Z/2$-graded Lie algebra of type (b).
\item If $\eps \into \Lambda^2 \rho$ then $\rho(\mathcal{L}_{\eps}(G))$ is a classical $\Z/2$-graded Lie algebra of type (c).
\end{enumerate}
\end{prop}
\begin{proof}
First consider the nondegenerate symmetric bilinear form induced by $\un \into \rho \otimes \rho$, and denote
$m \mapsto m^+$ the adjonction in $\End(V_{\rho})$. When $a \in A$ and $b \in A^{\dagger}$, $\rho(a-a^{-1})^+ = - \rho(a - a^{-1})$
and $\rho(b+b^{-1})^+ = \rho(b+b^{-1})$. Identifying $V_{\rho}$ with $\kk^N$ with its
the canonical bilinear form (this can be done up to an extension of scalar, which is harmless here) this means that
$\rho(\mathcal{L}_1(A^{\dagger}))$ is made of symmetric matrices while $\rho(\mathcal{L}_1(A))$ is made
of skew-symmetric matrices.
This implies
(1) and also (2), since in this case $\rho(\mathcal{L}_{\eps}(G)) = \gl(V_{\rho})$. In cases (3) and (4),
we apply lemma \ref{lemclifreal}. We have $\Res_A \rho = \varphi + \xi$,
$V_{\rho} = V_{\varphi} \oplus V_{\xi}$. We consider the nondegenerate bilinear form $<\ , \ >$ induced
by $\eps \into \rho \otimes \rho$. In case (3), it induces a blinear form on $V_{\varphi} \otimes V_{\xi}$
which is $A$-invariant hence zero, as $\xi \not\simeq \varphi^*$ (lemma \ref{lemclifreal}). As a consequence
$V_{\varphi} \perp V_{\xi}$ hence $< \ , \ >$ restricts to a nondegenerate $A$-invariant bilinear form on both
$V_{\varphi}$ and $V_{\xi}$, and we have $\rho(\mathcal{L}_1(A)) \subset \so(V_{\varphi}) \times \so(V_{\xi})$.
On the other hand, when $b \not\in A$, we have $b V_{\varphi} = V_{\xi}$ and $b V_{\xi} = V_{\varphi}$,
as $bV_{\varphi}$ is $A$-stable and $b V_{\varphi} \neq V_{\varphi}$ (otherwise $V_{\rho}$ would be reducible).
In matrix form, this implies $\rho(\mathcal{L}_1(A^{\dagger})) \subset \left\lbrace \left( \begin{array}{cc} 0 & -m^+ \\ m & 0
\end{array} \right) \right\rbrace$, with the notations used in describing type (b). Counting dimensions
we get that inclusions are equalities and
this implies (3). Case (4) is similar, except that $\varphi$ and $\xi$ have complex type,
hence $\xi \simeq \varphi^*$ (as $\rho \simeq \rho^*$), and this time the induced bilinear forms on
$V_{\varphi} \otimes V_{\varphi}$
and
$V_{\xi} \otimes V_{\xi}$ are zero, whereas the ones on $V_{\rho} \otimes V_{\xi}$ and $V_{\xi} \otimes V_{\rho}$
are nondegenerate. The identification with type (c) is then similar and straightforward.
\end{proof}

Since the representations of a Coxeter group are always of real type, this has the following consequence.

\begin{cor} \label{corZ2LeWcox} If $W$ is a finite Coxeter group and $\eps : W \onto \{ \pm 1 \}$ is the sign character 
then,
for $i \in \{ 0 , 1 \}$,
$$
\mathcal{L}_{\eps}(W)^i = 
 \left( \bigoplus_{\rho \in \mathcal{E}/\sim}  \gl(V_{\rho})^i \right)
\oplus \left( \bigoplus_{\rho \in \mathcal{F}} \osp(V_{\rho})^i \right)
$$
where the decompositions  $\gl(V_{\rho})^i$ (resp. $\osp(V_{\rho})^i$) are given by
the description as classical $\Z/2$-graded Lie algebras of types (a),(b),(c). 
\end{cor}

\begin{prop} \label{propL1Gcompl5}If $\rho^* \not\simeq \rho$ and $\rho^* \otimes \eps \simeq \rho$,
then
\begin{enumerate}
\item $\rho( \mathcal{L}_1(A)) = \rho( \mathcal{L}_{\eps}(G)) = \osp(V_{\rho})$
\item $\rho( \mathcal{L}_1(A)) \cap \rho( \mathcal{L}_1(A^{\dagger})) \neq 0$ unless $\forall b \in A^{\dagger} \ \rho(b^{-1}) = - \rho(b)$ 
\item We have $(\rho \oplus \rho^*)(\mathcal{L}_{\eps}(G)) = \osp(V_{\rho}) \oplus \osp(V_{\rho^*})
\simeq \osp(V_{\rho}) \oplus \osp(V_{\rho})$
with $(\rho \oplus \rho^*)(\rho(\mathcal{L}_1(A)) \simeq \{ (x,x) \ | \ x \in \osp(V_{\rho}) \}$
$(\rho \oplus \rho^*)(\rho(\mathcal{L}_1(A^{\dagger})) \simeq \{ (x,-x) \ | \ x \in \osp(V_{\rho}) \}$
\item As a $\Z/2$-graded Lie algebra, $(\rho \oplus \rho^*) (\mathcal{L}_{\eps}(G))$
is the double
of $\osp(V_{\rho})$.
\end{enumerate}
\end{prop}
\begin{proof}
Since $\rho^* \otimes \eps \simeq \rho$ we have a nondegenerate $W$-invariant bilinear
form on $V_{\rho}$ afforded by $\eps \into \rho \otimes \rho$, and $\rho(\mathcal{L}_{\eps}(G)) = \osp(V_{\rho})$.
Since  $\rho^* \not \simeq \rho$ we have $\rho \not\simeq \rho\otimes \eps$
hence $\Res_A \rho$ is irreducible. Moreover $\Res_A \rho$ is selfdual hence, according to \cite{LIEALG},
$\rho( \mathcal{L}_1(A)) = \osp(V_{\rho})$, which proves (1).
For all $g \in G$ we have $\eps(w) \rho(w^{-1}) = \rho(w)^+$ hence $\rho(b)^+ = - \rho(b)$ for $b \in A^{\dagger}$,
thus
$\rho( \mathcal{L}_1(A)) \cap \rho( \mathcal{L}_1(A^{\dagger})) = 0 \Leftrightarrow \rho(\mathcal{L}_1(A^{\dagger})) = 0$ means
$ \forall b \in A^{\dagger} \ \ \rho(b) = - \rho(b^{-1})$ which proves (2).
Now $\rho^*$ can be defined on $V_{\rho}$ as $g \mapsto \rho(g^{-1})^+$, and then
for $x,y \in V_{\rho}$ we have $< \rho^*(g)x,\rho^*(g) y> = < \eps(g) \rho(g) x, \eps(g) \rho(g) y >
 = < \rho(g) x, \rho(g) y > = \eps(g) <x,y >$.
 This proves $(\rho \oplus \rho^*)(\mathcal{L}_{\eps}(G) = \osp(V_{\rho}) \oplus \osp(V_{\rho^*})
\simeq \osp(V_{\rho}) \oplus \osp(V_{\rho}$ ; since $\rho(a)^+ = \rho(a^{-1})$ for $a \in A$
and $\rho(b)^+ = - \rho(b^{-1})$ for $b \in A^{\dagger}$, (3) follows, and (4) is an immediate consequence of (3).
\end{proof}

\begin{remark} When $W$ is a reflection group and $\eps$ the determinant, then : (1) one cannot have $\rho(b^{-1}) = - \rho(b)$
for all $b \in A^{\dagger} = \OO^{\dagger}$, because the reflections provide involutions in $A^{\dagger}$ ; (2) one may have $\rho \in \Irr(W)$ with $\rho^* \not\simeq \rho$ and $\rho^* \otimes \eps \simeq \rho$; an example is given by $W = G(3,3,4)$, and $\rho$ the restriction of the representation classically denoted
$([1,1],[2],\emptyset)$ of $G(1,1,4)$ (it can be checked more precisely, using character tables, that in this example $\eps \into S^2 \rho$)
\end{remark}

\begin{remark} An example of a triple $(G,\eps,\rho)$ as in the proposition with $\forall b \in A^{\dagger} \ \rho(b^{-1}) = - \rho(b)$is given by $G = \Z/4$, $\rho : G \to \GL_1(\C)$, $\eps : 1 \to -1$, $\rho : 1 \mapsto \sqrt{-1}$. There are no example in higher dimension, as proved by the following proposition.
\begin{prop}
Let $G$ be a finite group endowed with $\eps : G \to \{ \pm 1 \}$. If $\rho \in \Irr(G)$ satisfies $\rho^* \not\simeq \rho$
and $\forall g \in G \ \eps(g) = -1 \Rightarrow \rho(g^{-1}) = -\rho(g)$ then $\dim \rho = 1$.
\end{prop}
\begin{proof}
Let $A = \Ker \eps \vartriangleleft G$, and assume we have such a $\rho$.  Up to
considering $G/A \cap \Ker \eps$ one can assume that $\rho_{|A}$ is faithful. Since $g^2 \neq 1$ for $g \in G \setminus A$
this implies that $\rho$ is faithful. It follows that there exists $a_0 \in A$ of order 2 with $x^2 = a_0$ for all $x \in G \setminus A$.
From $\rho(a_0) = -1$ we also get $a_0 \in Z(G)$. Let $N = \dim V$. Using the Frobenius-Schur indicator, the
assumption $\rho^* \not\simeq \rho$ then translates into $\sum_{g \in A} \chi(g^2) = N |A|$ for $\chi = \tr \rho$, hence $\chi(g^2) = N$
that is $\rho(g^2) = 1$ hence $g^2 = 1$ for all $g \in A$. In particular $A$ is abelian. For $x,y \in G \setminus A$,
$xyx^{-1}y^{-1} = xyxa_0 ya_0 = (xy)^2 = 1$ and for $x \in G\setminus A$, $y \in A$, $xy \in G\setminus A$
hence $xyx^{-1}y^{-1} = xyxa_0 y = (xy)^2a_0 = a_0^2 = 1$. This proves that $G$ is abelian hence $\dim \rho = 1$.
\end{proof}
\end{remark}

Using lemma \ref{lemclifquat}, the proof of the following proposition is analogous to proposition \ref{propL1Greal}.
\begin{prop} \label{propL1Grquat}If $\rho$ has quaternionic type,
 then
\begin{enumerate}
\item $\rho( \mathcal{L}_1(A)) \cap \rho( \mathcal{L}_1(A^{\dagger})) = 0$.
\item If $\rho^* \otimes \eps \not\simeq \rho$ then $\rho(\mathcal{L}_{\eps}(G))$ is a classical $\Z/2$-graded Lie algebra of type (f).
\item If $\eps \into S^2 \rho$ then $\rho(\mathcal{L}_{\eps}(G))$ is a classical $\Z/2$-graded Lie algebra of type (d).
\item If $\eps \into \Lambda^2 \rho$ then $\rho(\mathcal{L}_{\eps}(G))$ is a classical $\Z/2$-graded Lie algebra of type (e).
\end{enumerate}

\end{prop}

We can apply the above result to the case of a complex reflection group $W$
with rotation subgroup $\OO$. We get the following.

\begin{prop} \label{propimageorthL1} If $\rho$ is a linear representation of $W$ with $\rho^* \simeq \rho$ then, extending $\rho$ into a $\kk(h^2)$-linear map,
$$
\begin{array}{lcl}
\rho(\HG2^0) &=& \kk(h^2) \left( \rho(\mathcal{H}) \cap \rho(\mathcal{L}_1(\OO)) \right) \\
\rho(\HG2^1) &=& h\kk(h^2) \left( \rho(\mathcal{H}) \cap \rho(\mathcal{L}_1(\OO^{\dagger})) \right) \\
\end{array}
$$
\end{prop}
\begin{proof}
This is a consequence of propositions \ref{propL1Greal} (1) and \ref{propL1Grquat} (1), as $\kk(h^2) \mathcal{H} = \HG2^0 \oplus h^{-1} \HG2^1 \subset \kk(h^2)W$,
and $\rho(\HG2^0) \subset \kk(h^2) \rho(\mathcal{L}_1(A))$, $\rho(\HG2^1) \subset h  \kk(h^2) \rho(\mathcal{L}_1(A^{\dagger}))$.

\end{proof}

Let us now assume that $W$ is a Coxeter group. In \cite{IH2}, a decomposition of $\mathcal{H}$ is obtained,
that we recall now, in the special case of a Coxeter group. For this we recall from \cite{IH2}
the notation
$$\mathrm{X}(\rho) = \{ \eta \in \Hom(W,\{ \pm 1 \}) \ | \ \forall s \in \mathcal{R} \ \eta(s) = -1 \Rightarrow
\rho(s) = \pm 1 \}$$
and also that $\mathrm{Ref}(W)$ is the set of all $\rho \in \Irr(W)$ such that, for all $ s \in \mathcal{R}$ with $\rho(s) \neq \pm 1$,
$\rho(s)$ is a reflection, and 
$$
\begin{array}{lcl} \mathrm{QRef} &=& \{ \eta \otimes \rho \ | \ \rho \in \mathrm{Ref}, \eta \in \Hom(W, \{ \pm 1 \}) \} \\
 \mathrm{\Lambda Ref} &=& \{ \eta \otimes \Lambda^k \rho \ | \ \rho \in \mathrm{Ref}, \eta \in \Hom(W, \{ \pm 1 \}) , k \geq 0 \} \\
\end{array}
$$
We also recall that there is an equivalence relation $\approx'$ on $\Irr(W)$ which, when $W$ has
no component of type $H_4$, is defined by
$$
\rho_1\approx' \rho_2 \Leftrightarrow \rho_2 \in \{ \rho_1 \otimes \eta, \rho_1^* \otimes \eta \otimes \eps \ | \ \eta
\in \mathrm{X}(\rho_1) \}
$$
(we refer the reader to \cite{IH2} for more details on the $H_4$ case).
Letting $\mathcal{R}/W$ denote the set of conjugacy classes of reflections, and
$\Irr'(W) = \Irr(W) \setminus \mathrm{\Lambda Ref}$,
theorem 1 of \cite{IH2} provides an explicit isomorphism
$$
\mathcal{H} \simeq \kk^{\mathcal{R}/W} \oplus \left( \bigoplus_{ \rho \in \mathrm{QRef}/\approx'}
\sl(V_{\rho}) \right)
\oplus \left( \bigoplus_{ \rho \in \mathrm{Irr}'(W)/\approx'}
\osp(V_{\rho}) \right)
$$
A consequence of proposition \ref{propimageorthL1} and corollary \ref{corZ2LeWcox} is that this isomorphism is $\Z/2\Z$ graded, thus providing
the following upgrade of theorem 1 of \cite{IH2}, in the case of a Coxeter group.

\begin{theor}\label{theodecHG2cox} If $W$ is a Coxeter group, then, as ($\Z/2\Z$)-graded Lie algebras,
we have an isomorphism, for $L = \kk(h^2)$ and $V_{\rho}^L = V_{\rho} \otimes_{\kk} L$,
$$
\HG2 \simeq L^{\mathcal{R}/W} \oplus \left( \bigoplus_{ \rho \in \mathrm{QRef}/\approx'}
\sl(V_{\rho}^L) \right)
\oplus \left( \bigoplus_{ \rho \in \mathrm{Irr}'(W)/\approx'}
\osp(V_{\rho}^L) \right)
$$
where the $\Z/2\Z$-grading on the $\sl(V_{\rho}^L)$ and the $\osp(V_{\rho}^L)$ are given
by the description as $\Z/2\Z$-graded Lie algebras of types (a),(b),(c), and $L$ is the 1-dimensional commutative
$\Z/2\Z$-graded Lie algebra of odd degree.
\end{theor}

\section{Zariski closures}
\label{sectionZar}

Our goal is to provide interpretations of the $(\Z/2\Z)$-grading in terms of Zariski closures of the braid groups inside the Hecke
algebra representations.
We let $\kk$ denote a field of characteristic $0$, $R = \kk[[h]]$, $K = \kk((h))$.
The setting of Chevalley in \cite{CHEV} is mostly
convenient for us, as we are dealing with subgroups of $\GL_m(F)$ for $F$ a field of characteristic $0$. In the sequel,
the notions that we use constantly refer to this setting. One basic lemma that we will need is the following one.

\begin{lemma} \label{lembasicgalg} Let $\Gamma$ be a Zariski-closed subgroup
of $GL_N(K)$ and $\Lie \Gamma$ its Lie algebra over $K$. For all
$x \in M_N(R)$, if $\exp(hx) \in \Gamma$ then 
\begin{enumerate}
\item $\exp(T hx) \in \Gamma(K[[T]])$.
\item for all $u \in \kk$, $\exp(u hx) \in \Gamma(K)$.
\item $x \in \Lie \Gamma$.
\end{enumerate}
Moreover, the algebraic closure of $\exp(hx)$ inside $\GL_N(K)$ is connected. More generally,
if $L$ is a subfield of $K$, then the algebraic closure of $\exp(hx)$ inside $\GL_N(K)$ considered as an $L$-group
is connected.
\end{lemma}
\begin{proof}
Let $a_1,\dots,a_r$ be polynomial functions on $M_N(K)$ with coefficients
in $R$ such that $\Gamma = \{ m \in GL_N(K) \ | \ \forall i \in [1,r] \ \ 
a_i(m)=0 \}$. Since $\Gamma$ is a subgroup of $GL_N(K)$, one has
$a_i(\exp(nhx))=0$ for all $i \in [1,r]$ and $n \in \Z$. Let $Q_i
= a_i(\exp(T h x)) \in (\kk[T])[[h]]$. Since $\kk$ has caracteristic
0, $\Z$ is Zariski-dense in $\kk$ hence $Q_i=0$ and $\exp(T h x)$
is a $\kk[[T]]$-point of $\Gamma$. This shows (1) and (2). It follows from (1) and \cite{CHEV}
ch. 2 \S 12 th\'eor\`eme 7 that $h x \in \Lie \Gamma$ hence $x \in \Lie \Gamma$, which is (3).
We now let $L$ denote a subfield of $K$, we let $\Gamma$ denote the algebraic closure of $\exp(hx)$ inside $\GL_N(K)$
considered as an $L$-group, and show that it is connected.
We can assume $x \neq 0$, for otherwise the statement is trivial.
Let $G = \{ \exp(nhx), n \in \Z \} \subset \Gamma$ be the cyclic subgroup generated by $X = \exp(hx)$,
$\Gamma_0$ be the connected component of the identity in $\Gamma$, and let $\pi : \Gamma \onto \Gamma/\Gamma_0$.
By the same as argument as for (2) we know that $\Gamma$ contains $G_{\kk} = \{ \exp(uhx), u \in \kk \} \simeq \kk$. Since $\Gamma/\Gamma_0$
is finite the restriction of $\pi$ to $G_{\kk}$ is trivial, hence $G \subset \Ker \pi = \Gamma_0$. Since $\Gamma$
is assumed to be minimal this proves $\Gamma = \Gamma_0$.
\end{proof}

Let now $\sigma \in Aut(K)$ be $f(h) \mapsto f(-h)$. Letting $x = h^2$
we have $R^{\sigma} = \kk[[x]]$, $K^{\sigma} = \kk((x))$ and $\Gal(K/K^{\sigma})
\simeq \Z/2\Z$. For $N \geq 1$, $Mat_N(K) = Mat_N(K^{\sigma}) \oplus h Mat_N(K^{\sigma})$
can be embedded into $Mat_{2N}(K^{\sigma})$ as a (closed) $K^{\sigma}$-subalgebra through
$a+hb \mapsto \left( \begin{array}{cc} a & xb \\ b & a \end{array} \right)$.
The group $\GL_N(K)$ can be identified to the subset of elements of $\GL_{2N}(K^{\sigma})$ of
the above form,
which is clearly an algebraic subgroup of $\GL_{2N}(K^{\sigma})$. 

\begin{lemma} \label{lemzar} Let $G$ be a closed (algebraic) $K^{\sigma}$-subgroup of $\GL_N(K)$ and
$X = \exp(a + h b) \in \GL_N(K)$ for some $a,b \in \gl_N(R^{\sigma})$ with $a \equiv 0 \mod h$.
Then $a + hb \in \mathrm{Lie} G$.
\end{lemma}
\begin{proof}
First note that, since $a \equiv 0 \mod h $, $X$ is well-defined. By Chevalley's formal exponentiation
theory (\cite{CHEV}, t. 2, ch. 2, \S 12, thm. 7) this is equivalent to saying $\exp( u (a + hb)) \in G(K^{\sigma}[[u]])$. We consider $G$
as defined in $\GL_{2N}(K^{\sigma})$, and let $\alpha_1,\dots,\alpha_r \in R^{\sigma}[m_{11},\dots,
m_{2N,2N}]$ be defining equations for $G$. Inside $\GL_{2N}(K^{\sigma}[[u]])$,
$\exp( u(a+hb))$ is a $2N \times 2N$ matrix which can be written 
$$
\sum_{n=0}^{\infty} \frac{u^n}{n!} \left( \begin{array}{cc} a & xb \\ b & a \end{array} \right)^n =
\sum_{n=0}^{\infty} x^n c_n(u)
$$
with $c_n(u) \in Mat_{2N}(\kk[[u]])$. Let $M =  \left( \begin{array}{cc} a & xb \\ b & a \end{array} \right)$.
It is easily checked that $M^4 \equiv 0 \mod x$, so we get that, modulo $x^r$,
$$
\sum_{n=0}^{\infty} \frac{u^n}{n!} M^n \equiv \sum_{n \leq 4r } \frac{u ^n}{n!} M^n
$$
and this shows that $c_n(u) \in Mat_{2N}(\kk[u])$. In particular,
$\alpha_i( \exp(u (a+hb))) = \sum_j c_{i,j}(u) x^j \in \kk[u][[x]]$ ; 
since $X^n \in G(K^{\sigma})$ for all $n \in \Z$ we have $\forall n \in \Z \ c_{i,j}(n) = 0$,
hence $c_{ij} = 0$, which proves the result.
\end{proof}

\begin{remark} In the situation above, $\dim_{K^{\sigma}} \mathrm{Lie} G = \dim _K K \otimes_{K^{\sigma}} \mathrm{Lie} G$,
and $K \otimes_{K^{\sigma}} \mathrm{Lie} G$ is the Lie algebra of the Zariski closure of $G$
in $K$ (\cite{CHEV}, t. 2, ch. 2, \S 8, prop. 2).
\end{remark}

Let $\Psi : B \to \GL_N(K)$ a representation of the Hecke algebra constructed 
from $\rho : W \to \GL_N(\kk)$ by monodromy (in which case $\kk = \C$) or through
generalizations of Drinfeld associators (see section \ref{sectortho} below),
and $\g$ the Lie subalgebra of $\gl_N(R)$ generated by the $hs$ for $s \in \mathcal{R}$.
Let $\g_0$ be the subalgebra of $\gl_N(\kk)$ spanned by the brackets
of an \emph{even} number of reflections, and $\g_1$ be the one spanned
by 
the brackets
of an \emph{odd} number of reflections. We define $\tilde{\g}_0 = \g_0 \otimes_{\kk} R^{\sigma}$,
$\tilde{\g}_1 = \g_1 \otimes_{\kk} R^{\sigma}$. We have $\g \subset \tilde{\g} = \tilde{\g}_0 \oplus h \tilde{\g}_1$.

\begin{prop} \label{propZarLieP}
Let $\overline{\Psi(P)}$ denote the Zariski closure of $\Psi(P)$ inside the $K^{\sigma}$-group $\GL_N(K)$.
It is connected, and $\mathrm{Lie} \overline{\Psi(P)} = \tilde{\g} \otimes_{R^{\sigma}} K^{\sigma}$. Moreover $\g_0 \cap \g_1 = \{ 0 \}$.
\end{prop}

\begin{proof}
First note that, by the monodromy construction, every $g \in P$ is mapped to $\Psi(g) = \exp(hx)$ for some
$x \in \gl_N(R)$. By lemma \ref{lembasicgalg} (applied to $L = K^{\sigma}$) it follows that the Zariski closure of the subgroup generated by each $\Psi(g)$ for $g \in P$ is connected. This implies that $\overline{\Psi(P)}$ is connected (see \cite{CHEV} ch. 2, \S 14, th\'eor\`eme 14).
For each $s \in \mathcal{R}$, let $Y_s \in P$ with $\Psi(Y_s) = \exp(hs + \dots)$,
that is $\Psi(Y_s) = \exp(a+hb)$ with $a, b \in \gl_N(R^{\sigma})$,
$a \equiv 0 \mod x$, $b \equiv s \mod x$. By lemma \ref{lemzar} we get $y_s = a + hb \in 
\mathrm{Lie} \overline{\Psi(P)}$. Note that $y_s \in \gl_N(R^{\sigma}) \oplus h \gl_N(R^{\sigma})$
and $y_s \equiv hs \mod x $. Also note that the $y_s$ belong to $\tilde{\g}_0 \oplus h \tilde{\g}_1 = \tilde{\g}$.
We let $\mathcal{L}$ denote the subalgebra of $\tilde{\g} \otimes_{R^{\sigma}} K^{\sigma}$ that they
generate. As a $K^{\sigma}$-algebra, it admits a basis of elements of the form $h^{\eta} [s_1,\dots,s_r]$,
with $\eta = 1$ if $r$ is odd, $\eta = 0$ if $r$ is even. We have
$$
\frac{1}{x^{\lfloor \frac{r}{2} \rfloor}} h^{\eta} [y_{s_1},\dots,y_{s_r}] \in \tilde{g} \mbox{\ and \ }
\frac{1}{x^{\lfloor \frac{r}{2} \rfloor}} h^{\eta} [y_{s_1},\dots,y_{s_r}] 
\equiv
h^{\eta} [s_1,\dots,s_r] \mod x
$$
where the notation $[a_1,\dots,a_r]$ denotes the iterated Lie bracket of the $a_1,\dots,a_r$.
It then follows (e.g. by using the determinant) that these elements are linearly independent over
$K^{\sigma}$, hence $\dim_{K^{\sigma}} \mathcal{L} \geq \dim_{K^{\sigma}} \tilde{\g}  \otimes_{R^{\sigma}} K^{\sigma}$
and $\mathcal{L} = \tilde{\g} \otimes_{R^{\sigma}} K^{\sigma}$.
Now the dimension of $\mathrm{Lie} \overline{\Psi(P)}$ over $K^{\sigma}$ equals the dimension $d$ of the
Zariski closure of $\Psi(P)$ in $\GL_N(K)$ over $K$ (\cite{CHEV} t. 2, ch. 2, \S 6 prop. 5),
hence 
$$\dim_{K^{\sigma}} K^{\sigma} \tilde{\g}_0 + \dim_{K^{\sigma}} K^{\sigma} \tilde{\g}_1 = \dim_{K^{\sigma}} \mathcal{L} \leq \dim_{K^{\sigma}} \mathrm{Lie} \overline{\Psi(P)} = d.$$ On the other hand,
we proved in \cite{IH2} that $d$ is
equal
to $\dim_{\kk} \rho(\mathcal{H})$. Since $d = \dim_{\kk} (\g_0 + \g_1) \leq \dim_{\kk} \g_0 + \dim_{\kk} \g_1 = 
\dim_{K^{\sigma}} K^{\sigma}\tilde{\g}_0 
+
\dim_{K^{\sigma}} K^{\sigma}\tilde{\g}_1$, we get $\dim_{K^{\sigma}} \mathrm{Lie} \overline{\Psi(P)} = \dim_{K^{\sigma}} \tilde{\g} \otimes_{R^{\sigma}} K^{\sigma}$, and $\dim( \g_0 + \g_1) = \dim \g_0 + \dim \g_1$, 
which concludes the proof.

\end{proof}

\begin{remark}
The latter statement $\g_0 \cap \g_1 = \{ 0 \}$ is a consequence of propositions \ref{propL1Greal} and \ref{propL1Gcompl5} if $\mathrm{tr} \rho$
only has real values. Combined with propositions \ref{propL1Gcompl5}  and \ref{propL1Grquat}, it proves on the other hand that, when $\rho^* \not\simeq \rho$,
$\rho^* \otimes \eps \simeq \rho$ and $\dim \rho > 1$, then $\g_1 \varsubsetneq \rho(\mathcal{L}_1(\OO^{\dagger}))$. 

\end{remark}

This proves  the first part of theorem \ref{theo16}, the second one being a consequence of proposition \ref{propimageorthL1}. This also proves
theorem \ref{theoZarHecke}, as follows.

\begin{cor} Under the weak BMR conjecture for $W$, the Zariski closure of $B$ inside $H(q)^{\times}$ as an algebraic
$\kk((h^2))$-group has for Lie algebra $\HG2 \otimes_{\kk(h^2)} \kk((h^2))$.
\end{cor}
\begin{proof}
Since $P$ has finite index in $B$, we can instead consider the Zariski closure of $P$. We identify
$H(q)^{\times}$ to a closed subgroup of $\GL_N(K)$ with $N = |W|$ and $K = \kk((h))$, and apply proposition \ref{propZarLieP}. In this
setting $\g_0 \otimes_{\kk} \kk(h^2) = \HG2^0$ and $\g_1 \otimes_{\kk} \kk(h^2) = h^{-1} \HG2^1$
hence $\tilde{\g}\otimes_{R^{\sigma}} K^{\sigma} = \HG2 \otimes_{\kk(h^2)} \kk((h^2))$ and this proves
the statement.
\end{proof}

Let $(C^r P)_{r \geq 0}$ denote the lower central series of $P$, let $N \otimes \kk$ denote
the $\kk$-Malcev completion of the nilpotent group $N$, and $P(\kk)$ the inverse
limit of the $(P/C^r P) \otimes \kk$, $r \geq 0$. One has $P(\kk) \simeq \exp \widehat{\mathcal{T}}$, where $\mathcal{T}$
is the holonomy Lie algebra of the hyperplane complement associated to $W$ (defined over $\kk$)
generated by the elements $t_s, s \in \mathcal{R}$, whose linear span constitute the homogeneous part
of degree $1$ of the graded Lie algebra $\mathcal{T}$. This part $\mathcal{T}^1$ can be canonically
identified with the first homology group $H_1(X_W,\kk)$. We let
$\widehat{\mathcal{T}}$ denote its completion w.r.t. the natural grading. The morphism
$\Psi : P \to \GL_N(K)$ can be extended to $\Psi_+ : P(\kk) \to \GL_N(K)$.

However, we have the following, which provides an interpretation of $\HG2^0$ as an algebraic Lie algebra.
We let $\HGR^{\mathrm{even}} = \bigoplus_r \HGR^{(2r)} = \bigoplus_r h^{2r} \HGR^{2r}$, and let $\widehat{\HGR}$, $ \widehat{\HGR}^{\mathrm{even}}$ denote the completions of $\HGR$ and $\HGR^{\mathrm{even}}$ with respect to the grading (or, equivalently : w.r.t. the $h$-adic topology).
\begin{prop} {\ }\label{propimpsiplus} Assume that $\rho : W \to \GL(\kk W) = \GL_N(\kk)$ is the regular representation of $W$. 
\begin{enumerate}
\item $\Imm \Psi_+ = \exp \widehat{\HGR}$
\item $\Psi_+(P(\kk)) \cap \GL_N(K^{\sigma})  = \exp \widehat{\HGR}^{\mathrm{even}}$
\item The Zariski closure of $\Psi_+(P(\kk))\cap \GL_N(K^{\sigma})$ inside $\GL_N(K^{\sigma})$ is connected,
and has $K^{\sigma}\HG2^0$ for Lie algebra.
\end{enumerate}
\end{prop}
\begin{proof}
By construction of $\Psi$ we have $\Psi_+(P(\kk)) \subset \exp \HGR$. Moreover, by construction
$P(\kk) = \exp \mathfrak{P}(\kk)$, where $\mathfrak{P}(\kk)$ is the inverse limit of the $\mathfrak{P}_r$,
with $\mathfrak{P}_r$ the nilpotent Lie $\kk$-algebra defined by $\exp \mathfrak{P}_r = (P/C^r P) \otimes \kk$.
Thus $\Psi_+ : P(\kk) \to \exp \HGR$ is surjective if and only if $\mathrm{d} \Psi_+ : \mathfrak{P}(\kk) \to \HGR$
is surjective. The morphism $\mathrm{d} \Psi_+$ respects the natural filtrations of $\mathfrak{P}(\kk)$ and $\HGR$,
hence it is surjective as soon as the associated morphism between the associated graded Lie algebras
$\gr \dd \Psi_+ :  \gr \mathfrak{P}(\kk) \to \gr \HGR = \HGR$ is surjective. Now $\gr \mathfrak{P}(\kk)$
is generated by $P^{ab} \otimes \kk \simeq H_1(X_W,\kk) \simeq \mathcal{T}^1$, and $\dd \Psi_+$
on $\mathcal{T}^1 \to \HGR^1$ is given by $t_s \mapsto hs$ for $s \in \mathcal{R}$, hence is surjective.
This proves (1). We have $\Psi_+(P(\kk)) = \exp \widehat{\HGR}$ by (1), and clearly $\exp \widehat{\HGR} \cap \GL_N(K^{\sigma}) = \exp \widehat{\HGR}^{\mathrm{even}}$, hence (2). For proving (3) it is thus sufficient to show that the Zariski closure $G$ of
$\exp \widehat{\HGR}^{\mathrm{even}}$ inside $\GL_N(K^{\sigma})$ is a connected $K^{\sigma}$-group, whose
Lie algebra is $K^{\sigma} \HG2^0$. First note that, for each $x \in \HGR^{2r}$, $\exp (h^{2r} x) \in G(K^{\sigma})$,
hence $x \in \Lie G$ by lemma \ref{lembasicgalg}.
This implies that $\Lie G$ contains $K^{\sigma} \HG2^0$. Now $\exp \widehat{\HGR}^{\mathrm{even}}$
is generated by elements of the form $\exp h x$ for $x \in \gl_N(\kk[[h]])$, which generate cyclic
subgroups whose Zariski closures are connected (see lemma \ref{lembasicgalg}). It follows that $G$ is connected
(\cite{CHEV} ch. 2, \S 14, th\'eor\`eme 14). Conversely, since $\kk[[h]]$ is a topological ring w.r.t. its
natural $h$-adic topology, the induced $h$-adic topology on $\gl_N(\kk[[h]])$ is thiner than
the Zariski topology. Because of this and because of lemma \ref{lembasicgalg}(2) it follows that $\Psi_+(P(\kk))
= \exp \widehat{\HGR}$ lies inside $\widehat{\Psi(P)}$,
hence $\exp  \widehat{\HGR}^{\mathrm{even}} \subset \widehat{\Psi(P)} \cap \GL_N(K^{\sigma})$ and
$G \subset  \widehat{\Psi(P)} \cap \GL_N(K^{\sigma})$. It follows (e.g. \cite{CHEV} t.2 ch. 2 \S 8, corollaire
de la proposition 1) that $\Lie G \subset (\Lie \widehat{\Psi(P)} ) \cap \Lie(\GL_N(K^{\sigma})) =
(K^{\sigma} \HG2) \cap \gl_N(K^{\sigma}) = K^{\sigma} \HG2^0$. This concludes the proof.

\end{proof}

We now establish the connection between the above statements and the existence of unitary structures.
For this we first need
to recall some well-known facts on unitary groups. In general, assume that $K/K_0$ is a quadratic extension, whose non-zero
automorphism is denoted $x \mapsto \bar{x}$. The corresponding
unitary group is $U_N = \{ g \in \GL_N(K) \ | \ ^t\overline{g} g = 1 \}$.
It is an algebraic group over $K_0$. For a $K_0$-algebra $S$,
its $S$-points are given by $U_N(S) = \{ g \in \GL_N(K \otimes_{K_0} S) \ | \ ^t\overline{g} g = 1 \}$,
where $x \mapsto \bar{x}$ is extended trivially on $K \otimes_{K_0} S$, and in particular $U_N = U_N(K_0)$. For $S = K$, we have
a canonical identification $K \otimes_{K_0} K \simeq K \oplus K$ given by
$x \otimes \la \mapsto (\la x, \la \overline{x})$, which identifies
$U_N$ with the subgroup of $\GL_N(K) \times \GL_N(K)$ made of the
couples $(g_1,g_2)$ such that $^tg_2 g_1 = 1$. This subgroup being isomorphic
(over $K_0$) with $\GL_N(K)$, we get that an isomorphism $U_N(K) \simeq \GL_N(K)$
which is an isomorphism of algebraic groups over $K$. In particular, $U_N$ is connected.

Let now $OSP_N(K) = OSP_N^A(K)$ be a closed subgroup of $\GL_N(K)$, described by
$\{ g \in \GL_N(K) \ | \ ^t g A g = A \}$ for some $A \in \GL_N(K_0)$.
Under the identification above, we get that $OSP_N(K)$
is isomorphic to the group of the $\{ (g_1,g_2) \in \GL_N(K) \ | \ ^t g_1 A g_1 = A,
^t g_2 A g_2 = A \}$, and $OSP_N(K) \cap U_N(K)$ is isomorphic to
$$
\begin{array}{ll}
 &\{ (g_1,g_2) \in \GL_N(K) \ | \ ^tg_2 g_1 = 1, ^t g_1 A g_1 = A, ^t g_2 A g_2 = A \}  \\
= &\{ (g_1,g_2) \in \GL_N(K) \ | \ ^tg_2 g_1 = 1, ^t g_1 A g_1 = A, g_1^{-1} A ^t g_1^{-1} = A \}  \\
=  &\{ (g_1,g_2) \in \GL_N(K) \ | \ ^tg_2 g_1 = 1, ^t g_1 A g_1 = A,  \, ^tg_1 A^{-1} g_1 = A^{-1} \}  \\
\simeq & OSP_N^A(K) \cap OSP_N^{A^{-1}}(K)
\end{array}
$$
Finally, if $OSP^A_N(K) \cap U_N(K)$ is known to act irreducibly on $K^N$, then Schur's lemma implies that there
is up to a scalar at most one bilinear form which is preserved, hence $OSP_N^A(K) = OSP_N^{A^{-1}}(K)$.
If $OSP^{A,0}_N(K)$ denotes the connected component of $OSP_N(K)$, it follows in this
case that $OSP^{A,0}_N(K) \cap U_N(K)$ is connected.

We come back to our setting of a monodromy representation $\Psi$ associated to some representation
$\rho$ of $W$, with notations $\g_0, \g_1$ as before.

\begin{center}
\emph{From now on, we assume that $\rho$ is an orthogonal representation.}
\end{center}

This condition is in particular satisfied in two important cases:
\begin{itemize}
\item when $W$ is a Coxeter group, or
\item when $\rho$ is the regular representation of $W$.
\end{itemize}
Under this assumption, we can assume that $ ^t \rho(s) \rho(s) = 1$, hence $^t \rho(s) = \rho(s)$ and $^t \sigma(\rho(hs)) + \sigma( \rho(hs)) = 0$ for all reflections $s$.
 This clearly implies $\rho(\HGR) \subset \mathfrak{u}_N(K)
 = \{ g \in \gl_N(K) \ | \ ^t \sigma(g) + g = 0 \}$. Now, recalling that $R = \kk[[h]]$, we have that $\rho(\HGR) \subset \tilde{\g}_0 \oplus h\tilde{\g}_1 \subset
 R  \rho(\mathcal{H}) =  \tilde{\g}_0 \oplus h \tilde{\g}_0 \oplus  \tilde{\g}_1 \oplus h \tilde{\g}_1$. Moreover, the
 elements $a \oplus hb \oplus c\oplus hd \in \tilde{\g}_0 \oplus h \tilde{\g}_0 \oplus  \tilde{\g}_1 \oplus h \tilde{\g}_1$
 belong to $\mathfrak{u}_N(K)$ if and only if $a + ^t a = 0$, $b- ^t b = 0$, $c + ^t c = 0$, $d - ^t d = 0$.
 Since $ m \in \g_0 \Rightarrow ^t m + m = 0$ and $ m \in \g_1 \Rightarrow ^t m - m = 0$, 
 this condition is equivalent to
  $b=c=0$.
 This means $(R\rho(\mathcal{H})) \cap \mathfrak{u}_N(K) = \tilde{\g}_0 \oplus h \tilde{\g}_1 = R^{\sigma} (\g_0 + h \g_1)$
hence $(K\rho(\mathcal{H})) \cap \mathfrak{u}_N(K) = K^{\sigma}(\g_0 \oplus h \g_1) \simeq \HG2 \otimes_{\kk(h^2)} \kk((h))$.

\begin{prop} \label{propintersectunitary}
Assume that $\rho$ is orthogonal. If $\widehat{\Psi(P)}$ denotes the Zariski closure of $\Psi(P)$
inside the $K$-group $\GL_N(K)$, then $\overline{\Psi(P)}$ is the connected component of the identity
of $\widehat{\Psi(P)} \cap U_N^{\sigma}(K)$.
\end{prop}
\begin{proof}
Clearly $\overline{\Psi(P)} \subset \widehat{\Psi(P)} \cap U_N^{\sigma}(K)$, and $\overline{\Psi(P)}$ is connected
by proposition \ref{propZarLieP}. 
Since
$(K\rho(\mathcal{H})) \cap \mathfrak{u}_N(K)  \simeq \rho(\HG2) \otimes_{\kk(h^2)} \kk((h))$
we get by proposition \ref{propZarLieP}
that these two $K^{\sigma}$-groups have the same Lie algebra.
The conclusion follows.
\end{proof}

\begin{remark} What happens exactly inside each irreducible representation,
in the non-orthogonal case, still requires further investigation.
\end{remark}

\begin{remark} In this section, the statements about $P$ also hold for the larger but less usual subgroup
of even braids, with the same proofs.
\end{remark}

\section{Orthogonal representations and palindromes}
\label{sectortho}

Assume $W$ is a Coxeter group, and let $\ss_1,\dots,\ss_n$ standard generators of the corresponding Artin
group $\pi_1(X/W,x_0)$, with $x_0$ in the chosen Weyl chamber. Then complex conjugation induces
an outer automorphism of $B$, known as the mirror image, which maps $\ss_i \mapsto \ss_i^{-1}$.

We state the following conjecture, which is a refinement of conjecture 1 in \cite{IH3}.

\begin{conj} \label{conjorth}
Let $(W,S)$ be a finite Coxeter system, $\kk$ a field of characteristic $0$. There exists
morphisms $\Phi : B \to W \ltimes \exp \widehat{\mathcal{T}}$, with $\mathcal{T}$ defined
over $\kk$, such that $\Phi(\ss)$ is conjugated to $s \exp t_s$ by some element
in $\exp \widehat{\mathcal{T}}_{\mathrm{even}}$ for every $s \in S$ and $\ss$ the generator associated to $s$. 
\end{conj}

Note that this conjecture implies conjecture 1 of \cite{IH3}, as every braided reflection
is conjugated to a simple generator by an element of $B$ that is a product of elements of
the form described in this conjecture. Also note that one cannot
expect this property to be true for arbitrary braided reflections, as it is not stable under conjugation.

This conjecture implies the following classical property, due to G. Lusztig (see \cite{LUSZ} , 1.7 ; see also \cite{GECK} \S 4 for a refinement), of the representations of Hecke algebras.

\begin{prop}\label{propsymlu}
Let $(W,S)$ be a finite Coxeter system such that \ref{conjorth} is true, and that $\rho : W \to \GL_N(\kk)$ 
be a representation in orthogonal form, meaning $\rho(w)^{-1} =\ ^t \rho(w)$ for all $w \in W$. Then the representation $\Psi$
of the generic Hecke algebra deduced from $\rho$ through $\Phi$ is
symmetric, meaning that $\ ^t \Psi(g) = \Psi(\tau(g^{-1}))$ for every $g \in B$.
\end{prop}
\begin{proof}
One only needs to show that $R(\ss)$ is a symmetric matrix, for $\ss$ a simple generator associated to $s \in S$. Let $H = \Ker(s-1)$. Since $\rho(s)$ is orthogonal and has
order 2, it is symmetric and so is $\varphi(t_H)$. It follows that $\rho(s) \exp h\varphi(t_H )$ is also symmetric, so it remains
to prove that the image of $\exp\widehat{\mathcal{T}}_{\mathrm{even}}$ is made of orthogonal matrices. We let $\sigma \in \Aut(K)$ be $h\mapsto - h$,$A\in\Aut(\gl_N(\kk))$ be $x\mapsto  - \ ^t x$
and $\tilde{\varphi} : \widehat{\mathcal{T}}\mapsto \gl_N(K)$ be defined by $t_H\mapsto h\varphi(t_H)$.
We have $A \circ \tilde{\varphi} (t_H) = - \ ^t h\varphi(t_H) = -h\varphi(t_H) = \sigma \circ \varphi(t_H)$, hence
$A\circ \tilde{\varphi} = \sigma \circ \tilde{\varphi}$ on $\widehat{\mathcal{T}}$ and $A\circ \tilde{\varphi} =\tilde{\varphi}$ on $\widehat{\mathcal{T}}_{\mathrm{even}}$. This implies that the elements of $\exp \widehat{\mathcal{T}}_{\mathrm{even}}$ are mapped to orthogonal matrices, and this concludes the proof.
\end{proof}

Recall that a consequence of the (in the Coxeter case weaker) conjecture 1 of \cite{IH3} is that, under the
same conditions, $\ ^t \Psi(g) = \sigma (\Psi(g^{-1}))$, hence $\sigma(R(g)) = R (\tau(g))$ for
all $g \in B$.

In case $W = \mathfrak{S}_n$
the following lemma  is an immediate consequence of Dehornoy's ordering and notions of $\ss_i$-positive braids.
In general, it is a straightforward consequence of the injectivity of the palindromization map of \cite{DELOUP}
(theorem 2.1 there).

\begin{lemma} \label{lemnopal} For $W$ a Coxeter group,
$\{ g \in B \ | \ \tau(g) = g \} = \{ 1 \}$.
\end{lemma}

\begin{theor} Conjecture \ref{conjorth} holds when $(W,S)$ has type $A_n$, $B_n = C_n$ or $I_2(m)$.
\end{theor}
In type $A$ the result is a consequence of the existence of even Drinfeld associators with rational coefficients.
Indeed, the isomorphism we need (up to rescaling $t_H \mapsto 2 t_H$ in $\mathcal{T}$) is
established in \cite{DRINFELD} (see proof of proposition 5.1) provided that the element $\varphi$ used there
belongs to the set denoted $M_1^+(\Q)$ of even rational associators. This is proposition 5.4 of
\cite{DRINFELD}.

In case $(W,S)$ has type $B_n$, a morphism $B \to W \ltimes \exp \widehat{\mathcal{T}}$ is associated in \cite{ENRIQUEZ}
proposition 2.3, to any element of a set $\mathrm{Pseudo}_{(\bar{1},1)}(2,\kk)$ of couples 
inside $(\exp \widehat{\mathcal{T}}^0) \times (\exp \widehat{\mathcal{T}}^1)$ where
$\mathcal{T}^0,\mathcal{T}^1$ denote the holonomy Lie algebras for the Coxeter types $A_2$ and $B_2$. It satisfies
our condition (again up to rescaling) when $(\Phi,\Psi) \in \mathrm{Pseudo}^+_{(\bar{1},1)}(2,\kk)$ where 
$$
\mathrm{Pseudo}^+_{(\bar{1},1)}(2,\kk) =  \mathrm{Pseudo}_{(\bar{1},1)}(2,\kk) \cap 
(\exp \widehat{\mathcal{T}}^0_{\mathrm{even}}) \times (\exp \widehat{\mathcal{T}}^1_{\mathrm{even}}).
$$
The conjecture in that case is thus a consequence of the following property.

\begin{prop}
$\mathrm{Pseudo}^+_{(\bar{1},1)}(2,\kk)  \neq \emptyset$
\end{prop}
\begin{proof}
In this proof, we freely use the notations of \cite{ENRIQUEZ}. The image of the complex conjugation $z \mapsto \bar{z}$
of $\Gal(\overline{\Q}|\Q)$ is mapped to $(-1,1) \in \widehat{\mathrm{GT}} \subset \widehat{\mathrm{GTM}}$. As
a consequence we have an element $(-1,-1,1,1) \in \widehat{\underline{\mathrm{GTM}}}$ of order $2$ which is mapped to
$A_- = (-1,-1,1,1) \in \underline{\mathrm{GTM}}(N,\Q) \subset  \underline{\mathrm{GTM}}(N,\Q_{\ell})$ for an arbitrary
prime number $\ell$, the map $\widehat{\underline{\mathrm{GTM}}} \to\underline{\mathrm{GTM}}(N)_{\ell} \to
\underline{\mathrm{GTM}}(N,\Q_{\ell})$ being described in \S 6.4 of \cite{ENRIQUEZ}.

We have $\mathrm{GTM}(N,\Q) \subset \Q^{\times} \times \Q(N) \times F_2(\Q) \times (\Ker \varphi_N)(\Q)$, and
$\Q(N) \simeq (\Z/N\Z) \times \Q$ with $-1 \mapsto (-\bar{1},-1)$. It is easily checked that the action of
$\Q(N)$ on $\exp \mathcal{T}_0$ and $\exp \mathcal{T}^1$ defined in \cite{ENRIQUEZ} associated
to $-1 \in \Q(N)$ the automorphism $t_H \mapsto -t_H$ exactly when $a \mapsto -a$ is the identity
of $\Z/N\Z$, that is when $N = 2$. Identifying $\mathrm{Pseudo}_{(\bar{1},1)}(N,\Q)$ with the quotients of
$\mathrm{Pseudo}(N,\Q)$ by $\Q(N)$, we get an action of $\mathrm{GTM}(N,\kk)$ on $\mathrm{Pseudo}_{(\bar{1},1)}(N,\kk)$
that we denote $\star$ (see \cite{ENRIQUEZ} \S 7.2 for an explicit formula). From the explicit description of the
action of $\mathrm{GTM}(N,k)$ on $\mathrm{Pseudo}(N,\kk)$ we get that
$\mathrm{Pseudo}_{(\bar{1},1)}^+(2, \kk) = \{ X \in \mathrm{Pseudo}_{(\bar{1},1)}(2, \kk) \ | \ A_- \star X = X \}$. Now, as in the proof of proposition
5.4 in \cite{DRINFELD}, we use the fact that there exists a map $\mathrm{Pseudo}_{(\bar{1},1)}^+(N,\kk) \to S_{\mathrm{alg}}(\kk)$, $(\Phi,\Psi) \mapsto
\Theta_{\Phi \Psi}$ with $S_{\mathrm{alg}}(\kk)$ the set of algebraic sections of the split exact sequence of $\Q$-group schemes
$1 \to \mathrm{GTM}_{(\bar{1},1)}(N,\kk) \to \mathrm{GTM}_{\bar{1}}(N,\kk) \to \kk^{\times} \to 1$, uniquely defined by $\Theta_{\Phi \Psi}(\la) \star (\Phi,\Psi) = (\Phi,\Psi)$ for all $\la
\in \kk^{\times}$, and with the property that its composition with the $S_{\mathrm{alg}}(\kk) \to \mathcal{S}(\kk)$, $\Theta \mapsto \mathrm{d} \Theta$,
where $\mathcal{S}(\kk)$ is the set of sections of the corresponding sequence of Lie algebras $0 \to \mathfrak{gtm}_{(\bar{1},1)}(N,\kk) \to \mathfrak{gtm}_{\bar{1}} (N,\kk) \to \kk \to 0$
is  a bijection. By conjugation we have an action $\Theta \mapsto \Theta^{A-} : \la \mapsto A_- \Theta(\la) A_-^{-1} = A_- \Theta(\la) A_-$
of $\Z/2\Z = < A_- >$ on $S_{\mathrm{alg}}(\kk)$, and it is easily checked that $\Theta_{\Phi \Psi}^{A_-} = \Theta_{\tilde{\Phi} \tilde{\Psi}}$ where
$\tilde{\Phi}$, $\tilde{\Psi}$ are deduced from $\Theta, \Psi$ through $t_H \mapsto - t_H$.
It follows that $\mathrm{Pseudo}^+_{(\bar{1},1)}(2,\kk) \neq \emptyset$ iff there exists $\Theta \in S_{\mathrm{alg}}(\kk)$ such that
$\Theta^{A_-} = \Theta$, iff there exists some $\mathrm{d} \Theta \in \mathcal{S}(\kk)$ such that $(\mathrm{d} \Theta)^{A_-} = \mathrm{d} \Theta$. Let
$c \in \Aut F_2$ be $x \mapsto x^{-1}$, $y \mapsto y^{-1}$ with $F_2 =<x,y>$. Using as before the notations of \cite{ENRIQUEZ}, $c$ induces an action 
on $\Ker \varphi_N$, $F_2(\kk)$ and $(\Ker \varphi_N)(\kk)$, for all $N \geq 2$. By explicit computations in $\mathrm{GTM}(2,\kk)$ we get
$A_-  (\la,\mu,f,g) A_- = (\la,\mu,c.f,c.g)$. Now, the Lie algebra $\mathfrak{gtm}_{\bar{1}}(N,\kk)$ of $\mathrm{GTM}_{\bar{1}}(N,\kk)$ is made of
triples $(s,\varphi,\psi)$ matching with tangent vectors $\la = 1 + \eps s$, $\bar{\mu} = 1$, $f = \exp \eps \varphi$,
$g = \exp \eps \psi$ with $\varphi \in \hat{\mathfrak{f}}_2(\kk) = \hat{\mathfrak{f}}(\xi,\eta)$, $\psi \in \hat{\mathfrak{f}}(\Xi,\eta(0),\eta(1),\dots,\eta(N-1))$
with $\xi = \log x$, $\eta = \log y$, $\Xi = \log x^N$, $\eta(\alpha ) = \log x^{\alpha} y x^{-\alpha}$, and the automorphisms
induced by $c$ are $\xi \mapsto -\xi$, $\eta \mapsto - \eta$, $\Xi \mapsto - \Xi$, $\eta(\alpha) \mapsto - \eta(-\alpha) = - \eta(\alpha)$
if $N = 2$. Thus $c$ induces $\varphi \mapsto \tilde{\varphi}$, and $\mathrm{Pseudo}^+_{(\bar{1},1)}(2, \kk) \neq \emptyset$ iff there exists
$(1,\varphi,\psi) \in \mathfrak{gtm}_{(\bar{1},1)}(2,\kk)$ with $\tilde{\varphi} = \varphi$, $\tilde{\psi} = \psi$. From $(1,\varphi,\psi) \in \mathfrak{gtm}_{(\bar{1},1)} (2,\kk)\neq \emptyset$
one builds $\frac{1}{2} ( (1,\varphi,\psi), + c. (1,\varphi,\psi))/2 = (1 + (\varphi + \tilde{\varphi})/2, (\psi + \tilde{\psi})/2)  \in \mathfrak{gtm}_{(\bar{1},1)} (2,\kk)$, which provides
a convenient section in $\mathcal{S}(\kk)$ and concludes the proof.

\end{proof}

In case $W$ is a dihedral group, the conjecture also holds because \cite{DIEDRAUX} (see \S6 there) provides the necessary material in
order to adapt Drinfeld's proof in a straightforward way : the analogue of proposition 5.2 of \cite{DRINFELD} is proved in
\S 6 for the group $G'(\kk)$ defined there, and even associators are the fixed elements of $\mathbbm{Ass}'_1(\kk)$ under the involution
$(-1,1) \in G'(\kk)$, whose action on $B(\kk)$ extends the `mirror automorphism' (aka complex conjugation) of $B$.

One actually gets this way an alternate proof of the following result of Lusztig (the original result
however do not need the restriction on exceptional types that we need to impose for now, and is thus stronger).

\begin{cor} Let $(W,S)$ be a finite irreducible Coxeter system not of the exceptional types $F_4,H_3,H_4,E_6,E_7,E_8$.
Then every representation $\Psi$ of $H(q)$
has a matrix model such that $\ ^t \Psi(g) = \Psi(\tau(g))$ for all $g \in B$.
\end{cor}
\begin{proof}
Except in type $D_n$, this corollary is a consequence of proposition \ref{propsymlu} and of the above results, as all representations
of $W$ can be realized over $\R$. Type $D$ can then be reduced to type $B$ by using inclusions between the corresponding
Hecke algebras, applying verbatim the arguments of \cite{IH3}, cor. 6.2.

\end{proof}

\begin{remark}
\label{remEXPLAIN} : 
Assuming that $\Phi$ satisfies conjecture \ref{conjorth} and that $\Psi$ is the representation of $H(q)$
deduced from $\Phi$ and some $\rho \in \Irr(W)$,
a  consequence of proposition \ref{propsymlu} and lemma \ref{lemnopal} is that, if $\Psi$ is faithful,
then $\Psi(P) \cap O_N(K^{\sigma}) = \{ 1 \}$, since $U_N^{\sigma}(K) \cap O_N(K) = O_N(K^{\sigma})$. On the other hand,
and using the notations $\Psi_+$, $P(\kk)$ of \S \ref{sectionZar}, we have $\Psi_+(P(\kk)) \subset U_N^{\sigma}(K)$
and $U_N^{\sigma}(K) \cap \GL_N(K^{\sigma}) = O_N(K^{\sigma})$. It then follows from proposition \ref{propimpsiplus} (3)
that $\Psi_+(P(\kk)) \cap O_N(K^{\sigma})$ is a large group, since it has for Zariski closure a connected
algebraic group whose Lie algebra is $K^{\sigma} \HG2^0$.
\end{remark}

\section{Rotation algebras : general facts}
\label{sectrot} 

Let $W$ be an irreducible 2-reflection group, $\eps : W \to \{ \pm 1 \}$ the sign morphism and $\OO = \Ker \eps$ the rotation subgroup. We
introduce the following Lie algebra, for $\kk$ a commutative unital ring.

\begin{defi} The infinitesimal rotation algebra $\mathcal{A}$ is the Lie subalgebra of $\kk \OO$
generated by the $[s,u] = su-us$ for $s,u \in \mathcal{R}$ and $su \neq us$.
\end{defi}

\subsection{Preliminaries on finite reflection groups}

We prove the following.

\begin{lemma} \label{lemgensu}
When $W$ is an irreducible complex 2-reflection group, then $\OO$ is generated by the $su$ for $s,u \in \mathcal{R}$ and $su \neq us$.
\end{lemma}
\begin{proof} 
Recall that, when $H$ is a subgroup of $G$ with $G$ generated by $a_1,\dots,a_r$, and $\alpha  : G/H \to G$
is a section of $G \mapsto G/H, x \mapsto \bar{x}$, then $H$ is generated by the $y a_i \alpha(\overline{y a_i})^{-1}$
for $y \in \alpha(G/H)$ and $ 1 \leq i \leq r$. In case $G = W$, $H = \OO$ and $\{ a_1, \dots , a_r \} = \mathcal{R}_0 \subset
\mathcal{R}$ an arbitrary generating set for $W$, we get that $\OO$ is generated by the $su$ (and $us = (su)^{-1}$)
for $s$ a fixed reflection and $u$ running among $\mathcal{R}_0$.  Since $W< \GL_n(\C)$ is irreducible,
there exists $s_1,\dots,s_n$ such that the graph $\Gamma$ with vertices $s_1,\dots,s_n$ and edges $s_i \mbox{---} s_j$
iff $s_i s_j \neq s_j s_i$ is connected, and such that the subgroup $W_0<W$ generated by the $s_1,\dots,s_n$ is irreducible (see e.g. \cite{ARRREFL}, prop. 3.1). 
We know
that $A$ is generated by the $s_1 u$ and $u s_1$ for $u \in \mathcal{R}$.  For a given $u$, because the action of $W_0$ is irreducible
and $W_0$ is generated by $s_1,\dots,s_n$,
there exists $1 \leq j \leq n$ such that $us_j \neq s_j u$.
By connectedness of $\Gamma$ there exists $i_1,\dots,i_k$ with $i_1 = 1$ and $i_k = j$ such that $s_{i_t} s_{i_{t+1}} \neq s_{i_{t+1}} s_{i_t}$
for $1 \leq t \leq k-1$. Then $s_1 u = (s_{i_1} s_{i_k})(s_{i_k} u) = (s_{i_1} s_{i_2})(s_{i_2} s_{i_3}) \dots (s_{i_{k-2}}s_{i_{k-1}})(s_{i_{k-1}}s_{i_{k}})(s_j u)$.
Using the same argument for $us_1$, this proves the lemma.
\end{proof}

In case of Coxeter groups, we actually have the following stronger version. We do not know whether it holds for an arbitrary
irreducible 2-reflection group equipped with an arbitrary generating set of reflections.
\begin{lemma} \label{lemgensufort}
When $W$ is an irreducible Coxeter group, then $\OO$ is generated by the $su$ for $s,u$ simple reflections with $su \neq us$.
\end{lemma}
\begin{proof} Every element of $\OO$ can be written as a product of an even number of simple reflections,
hence it is enough to prove that every $su$ for $s,u$ simple reflections can be written as a product
of $s_i s_{i+1}$ for $(s_i s_{i+1})^2 \neq 1$ and $s_i$, $s_{i+1}$ simple reflections. This is an immediate consequence of the connectedness
of the Coxeter graph.
\end{proof}

\begin{lemma} \label{lemgensu3geen} When $W$ is a complex reflection group of type $G(e,e,n)$ for $n \geq 3$, then $\OO$
is generated by the $su$ which have order 3, for $s,u \in \mathcal{R}$.
\end{lemma}
\begin{proof}
By lemma \ref{lemgensu} we know that $W$ is generated by the $su$ for $s,u \in \mathcal{R}$ and $su \neq us$.
We need to express such a $su$ as a product of elements as in the statement. Up to conjugation
by an element of $\mathfrak{S}_n$, we can assume that $s = S \oplus \mathrm{Id}_{n-3}$ and
$u= U \oplus \Id_{n-3}$ with 
\begin{itemize}
\item either $S = \left( \begin{array}{ccc} 0 & \zeta & 0 \\ \zeta^{-1} & 0 & 0 \\ 0 & 0 & 1 \end{array} \right)$,
$U = \left( \begin{array}{ccc} 1 & 0 & 0 \\ 0 & 0 & \eta \\ 0 & \eta^{-1} & 0 \end{array} \right)$ for some $\zeta ,\eta \in \mu_e$,
and then $SU$ has order 3
\item either $S = \left( \begin{array}{ccc} 0 & \zeta & 0 \\ \zeta^{-1} & 0 & 0 \\ 0 & 0 & 1 \end{array} \right)$,
$U = \left( \begin{array}{ccc} 0 & \eta & 0 \\ \eta^{-1} & 0 & 0 \\ 0 & 0 & 1 \end{array} \right)$,
and then, letting $\delta = \zeta \mu^{-1}$,
$$
SU =  \left( \begin{array}{ccc} \delta & 0 & 0 \\ 0 & \delta^{-1}  & 0 \\ 0 & 0 & 1 \end{array} \right) =
\left[ \left( \begin{array}{ccc} 0 & \delta & 0 \\ \delta^{-1} & 0 & 0 \\ 0 & 0 & 1 \end{array} \right)
 \left( \begin{array}{ccc} 1 &0 & 0 \\ 0 & 0 & 1 \\ 0 & 1 & 0 \end{array} \right)\right] \left[
 \left( \begin{array}{ccc} 1 &0 & 0 \\ 0 & 0 & 1 \\ 0 & 1 & 0 \end{array} \right)
 \left( \begin{array}{ccc} 0 &1 & 0 \\1 & 0 & 0 \\ 0 & 0 & 1 \end{array} \right) \right]
$$
is a product of two `$su$' of order 3.
\end{itemize}

\end{proof}

\begin{lemma} When $W$ is an irreducible complex 2-reflection group of exceptional type with a single reflection class, then 
$\OO$
is generated by the $su$ which have order 3, for $s,u \in \mathcal{R}$.
\end{lemma}
\begin{proof} Case-by-case check, using CHEVIE.
\end{proof}

\begin{remark} The two statements above do \emph{not} hold for the dihedral groups $G(e,e,2)$. For the
groups $G(2e,e,n)$ with $n \geq 3$, as well as for the exceptional groups with two reflection classes, 
the subgroup generated by the $su$ of order $3$ has order $4$ in $W$. For the exceptional group this
is a simple computer check; in the case of the $G(2e,e,n)$ for $n \geq 3$ this is because this subgroup $H$ is equal to the
rotation subgroup of the group $G(2e,2e,n)$. Indeed, it contains this rotation subgroup by lemma \ref{lemgensu3geen} and
conversely, because a reflection $s \in G(2e,e,n) \setminus G(2e,2e,n)$ is necessarily a diagonal matrix, one easily
checks that a product $su$ or $us$ for $u$ an arbitrary reflection in $G(2e,2e,n)$ has order $2$ or $4$, hence the generators
of $H$ are all contained in $G(2e,2e,n)$.

\end{remark}

Using the Shephard-Todd classification, this has for consequence the following.

\begin{prop} \label{propgenodd} If $W$ is a 2-reflection group with a single reflection class, then $\OO$ is
generated by the $su$ of odd order for $s,u \in \mathcal{R}$.
\end{prop}
\begin{proof}
When $W = G(e,e,n)$ for $n \geq 3$ or $W$ of exceptional types this
is a consequence of the above lemmas. The remaining cases of the dihedral groups $W = G(e,e,2)$
with $e$ odd is trivial, as $\OO$ is cyclic and admits as generator a product of two reflections.

\end{proof}

\begin{lemma} \label{lemsingleconjclass} If $W$ is a Coxeter group of type 
\begin{enumerate}
\item $A_n$ for $n \geq 4$, 
\item $D_n$ for $n \geq 5$
\item $E_6,E_7,E_8$, 
\end{enumerate}
then the
$su$ for $su \neq us$ constitute a single conjugacy class in $\OO$. 
\end{lemma}

\begin{proof}
Considering the elements of the Coxeter group of type $A_{n-1}$ (resp. $D_n$) as (signed) permutations, we
can associate to each element $g$ 
its support
$supp(g) = \{ i \ | \ g(i) \neq i \} \subset \{1 , \dots, n \}$. For $s,u \in \mathcal{R}$,
$su \neq us$ implies $supp(s) \cap supp(u) \neq \emptyset$,
that is $supp(s) \cup supp(u) = \{i,j,k\} \subset \{1,\dots, n\}$.
Choosing another couple $(s',u')$, the union of the supports
of the 4 reflections has size at most $6$,
hence up to conjugation by an element of $\mathfrak{S}_n$, we can assume that
$s,u,s',u' $ all belong to the Coxeter group of type $A_5$ (resp.$D_6$), at least if $n \geq 5$ (resp. $n \geq 6$).
This reduces the proof to the check of a few cases which are readily done.

\end{proof}

\begin{remark} There are 2 such conjugacy classes for types $A_3$ and $D_4$.

\end{remark}

\subsection{Preliminaries on commutative algebras}

\begin{lemma}\label{lempolgg}  Let $G$ be a finite group, $g \in G$ of order $N$, and $\kk$ a field of characteristic $0$. Then 

\begin{enumerate}
\item $g$ is a polynomial of $g-g^{-1}$
inside $\kk G$
if and only if $N$ is odd. In this case, the polynomial can be chosen with coefficients in $\Q$.
\item $g+ g^{-1}$ is an \emph{even} polynomial in $g-g^{-1}$ if and only if $N$ is odd.
In this case, the polynomial can be chosen with coefficients in $\Q$.
\end{enumerate}
\end{lemma}
\begin{proof}
We first prove (1). We can assume $\kk \subset \C$. Firstly we assume $\kk = \C$. This amounts to check whether $x$ is a polynomial in  $x-x^{-1}$ inside $\C[x]/(x^N -1)$. Letting $\zeta = \exp(2 \ii \pi/N)$
this means that there exists $P \in \C[X]$ such that $\forall k \ \ \zeta^k = P(\zeta^k - \zeta^{-k})$, which is possible
(Lagrange's criterium) exactly when the $\zeta^k - \zeta^{-k}$ are distinct integers for $0 \leq k < N$,
that is when $N$ is odd. In that case, $P$ has degree at most $N-1$, and we proved that the family
of the $(x-x^{-1})^r$ for $0 \leq r \leq N$
subspace of $\C[x]/(x^N -1)$
spanned by the $(x-x^{-1})^r$, $0 \leq r \leq N$, is a basis.

This proves in the general case that $g$ can be a polynomial in $g - g^{-1}$ only if $N$ is odd. Conversely, we only
need to prove that we can choose $P \in \Q[X]$. But we proved that the $(x-x^{-1})^r$ for $0 \leq r < N$
are linearly independent over $\C$ inside $\C[x]/(x^N-1)$, so they are also linearly independent inside
$\Q[x]/(x^N-1)$, hence they form a basis of $\Q[x]/(x^N-1)$ and one can indeed find $P \in \Q[X]$ with $x = P(x-x^{-1})$.

We now prove (2). We can assume $G = < g >$ (and so $G \simeq \Z/N$). If $g+g^{-1} = R(g-g^{-1})$ for
some $R \in \kk[X]$, then  $g = P(g-g^{-1})$ with $P(X) = (R(X) + X)/2$ hence $N$ is odd by (1). Conversely, we assume $N$
is odd. By (1), we have $g = P(g-g^{-1})$ for some $P \in \Q[X]$. There exists $\varphi \in \Aut(G)$ such that $\varphi(g) = g^{-1}$.
We extend it to an automorphism of $\kk G$, and get $g^{-1} = \varphi(g) = P(\varphi(g-g^{-1}))= P(g^{-1} - g) = Q(g-g^{-1})$
for $Q(X) = P(-X) \in \Q[X]$. Thus $g + g^{-1} = (P+Q)(g-g^{-1})$ with $P+Q \in \Q[X] $ an even polynomial.
\end{proof}

\begin{lemma} \label{lempolordre3} Let $G$ be a finite group, $g \in G$ of order $3$. Then, $\Ad(g) : x \mapsto g x g^{-1}$
is a polynomial in $\ad(g) - \ad(g^{-1}) : x \mapsto gx-xg - (g^{-1} x - x g^{-1})$.
\end{lemma}
\begin{proof} E.g. by direct computation, a convenient polynomial being $\frac{1}{24} X^4 + \frac{1}{12} X^3 + \frac{5}{8} X^2 + \frac{3}{4} X
+1$, or by using a similar argument as in the previous lemma.
\end{proof}

\subsection{Basic facts}

For every representation $\rho : \OO \to \GL(V)$ of $\OO$, we let $\rho_{\mathcal{A}} : \mathcal{A} \to \gl(V)$ denote the induced representation of $\mathcal{A}$.
\begin{prop}  \label{propirrA} {\ }  
\begin{enumerate} Assume $W$ is an irreducible 2-reflection group with a single reflection class.
\item $\AA$ generates $\kk \OO$ as an associative algebra with $1$.
\item $\forall \rho \in \Irr(\OO)$, $\rho_{\AA}$ is irreducible.
\item $\AA$ is reductive.
\item $\forall \rho^1,\rho^2 \in \Irr(\AA)$, $\rho^2 \simeq \rho^1 \Leftrightarrow \rho^2_{\AA} \simeq \rho^1_{\AA}$. 
\item $\forall \rho^1,\rho^2 \in \Irr(\AA)$, $\rho^2 \simeq (\rho^1)^* \Leftrightarrow \rho^2_{\AA} \simeq (\rho^1_{\AA})^*$. 
\end{enumerate}
\end{prop}
\begin{proof}
By proposition \ref{propgenodd}, $\OO$ is generated by the $su$ for $s,u \in \mathcal{R}$ of odd order.
Since $[s,u] = su - us = (su) - (su)^{-1}$, lemma \ref{lempolgg} (1)  implies $\rho(\mathsf{U} \AA) \supset \rho(\OO)$
hence (1) and (2).  Since $\C \OO = \bigoplus_{\rho \in \Irr(\OO)} \End(V_{\rho})$ and $\AA \subset \C \OO$,
$\bigoplus_{\rho \in \Irr(\OO)} \rho_{\AA}$ defines a faithful semisimple representation of $\AA$, hence (3).
Let now $\rho^1 , \rho^2 \in \Irr(\OO)$. Clearly $\rho^1 \simeq \rho^2 \Rightarrow \rho^1_{\AA} \simeq \rho^2_{\AA}$.
Conversely, we can assume $\rho^1,\rho^2 : \OO \to \GL(V)$, and let $Q \in \End(V)$ 
be such that $\rho^2(y) = Q \rho^1(y) Q^{-1}$ for all $y \in \AA$. 
In particular $\rho^2(x-x^{-1}) = Q \rho^1(x-x^{-1})Q^{-1}$ whenever $x = su$ with $s,u \in \mathcal{R}$.
If in addition $x$ has odd order there exists by lemma \ref{lempolgg} (1) a polynomial $P \in \Q[X]$ such that
$x = P(x-x^{-1})$, hence $\rho^2(x) = Q \rho^1(x) Q^{-1}$. Since such $x$ generate
$\OO$ we get  $\rho^2(g) = Q \rho^1(g) Q^{-1}$ for all $g \in \OO$ hence (4).
Finally, assuming $\rho^2 \simeq (\rho^1)^*$
means $\rho^2(x) = Q  \ ^t\rho^1(x^{-1})$ for all $x = su \in \OO$
we get $\rho^2(x-x^{-1}) = Q\ ^t \rho^1(x^{-1} -x) Q^{-1} = - Q \ ^t \rho^1(x-x^{-1}) Q^{-1}$ hence $\rho^2_{\AA} \simeq (\rho^1_{\AA})^*$.
Conversely, if $\rho^2_{\AA} \simeq (\rho^1_{\AA})^*$ then
$\rho^2(x-x^{-1}) = Q \ ^t \rho^1(x^{-1}-x)) Q^{-1}$. When $x$ has odd order we can (by lemma \ref{lempolgg} (2)) choose $P \in \Q[X^2]$ such that $x +x^{-1} = P(x-x^{-1})$
hence $\rho^2(x+x^{-1}) = Q \ ^t \rho^1(x +x^{-1}) Q^{-1}$.
Then $2 \rho^2(x) = \rho^2(x+x^{-1} + x - x^{-1}) = 2 Q \  ^t \rho^1(-x+x^{-1} + x + x^{-1}) Q^{-1} = 2 Q \ ^t \rho^1(x^{-1}) Q^{-1}$.
This yields $\rho^2 \simeq( \rho^1)^*$ by proposition \ref{propgenodd} hence (5).

\end{proof}

\begin{prop}
Let $W$ be a 2-reflection group.
Let $\rho \in \Irr(\OO)$ such that $\rho^* \simeq \rho$,  $\beta$ the corresponding nondegenerate bilinear form on $V_{\rho}$
and $y \mapsto y^+$ the adjunction operation w.r.t. $\beta$.
Then
$$\rho(\AA) \subset \osp(V_{\rho}) = \{ y \in \End(V_{\rho}) \ | \ \tr y = 0 \mbox{ and } y^+ = - y \}$$ 
\end{prop}
\begin{proof} For $x = su$ and $s,u \in \mathcal{R}$ we have
$\rho(x)^+ = \rho(x)^{-1} = \rho(x^{-1})$ hence $\rho(x-x^{-1})^+ = \rho(x^{-1} - x) = - \rho(x-x^{-1})$.
Moreover $\tr \rho(x)^+ = \tr \rho(x)$, hence $\tr \rho(x-x^{-1}) = 0$. Since $\AA$ is generated as
a Lie algebra by such $x -x^{-1}$ this concludes the proof.
\end{proof}

\begin{prop} \label{propsemisimpA} Let $W$ be a 2-reflection group with a single reflection class.
If the $su \in \OO$ for $su \neq us$ form a single conjugacy class in $\OO$, then
$Z(\AA) = \{ 0 \}$. Thus in that case $\AA$ is a semisimple Lie algebra.
\end{prop}
\begin{proof}
Let $L$ be the subspace of $\kk \OO$ spanned by the $[s,u]$ for $s,u \in \mathcal{R}$.
We have $\AA \subset L + [\AA,\AA]$.
Since $\AA$ generates $\kk \OO$ as an associative algebra with 1 by proposition \ref{propirrA}
we have $Z(\AA) \subset Z(\kk \OO)$. For $c$ a conjugacy class in $\OO$ we let
$\delta_c$ denote the associated central function. Then $[\AA,\AA] \subset
[ \kk \OO, \kk \OO ] = \bigcap_c \Ker \delta_c$. When $y \in Z(\AA) \subset Z(\kk \OO)$,
if $T_c = \sum_{g \in c} g$, then $y$ can be written $y = \sum \la_c T_c$
for some $\la_c \in \kk$. But $\delta_c(\AA) = \delta_c(L)$
and $\delta_c(L) = 0$ if $c$ is not the class $c_0$ formed by the $x = su$ for $su \neq us$,
$s,u \in \mathcal{R}_0$. Thus $y = \la_0 T_{c_0}$. Since $x^{-1} \in c_0$ we
have $ \delta_{c_0}(x-x^{-1}) = 0$ hence $y = 0$.

\end{proof}

\subsection{Small rank}

\subsubsection{$W = \mathfrak{S}_4$}

The irreducible representations of $\OO = \mathfrak{A}_4$ are the 1-dimensional
$[4]$, $[2,2]^+$, $[2,2]^-$ and the 3-dimensional $[3,1]$. On $[4]$
all the $[s,u]$ act by $0$. The 3-cycle $x = (1 \ 2 \ 3)$ acts by $j$ on $[2,2]^+$
and $j^{-1} = j^2$ on $[2,2]^-$ for $j$ a primitive 3-rd root of $1$.
Then $x -x^{-1} \in \AA$ acts by a non-zero value on $[2,2]^+$
by its complex conjugate on $[2,2]^-$. The same holds true for all
of the $x-x^{-1} \in \AA$ for $x$ a $3$-cycle. On the other hand, $[3,1]$
is a self-dual representation of $A$, hence the image of $\AA$
inside $\gl([3,1])$ is included in $\so([3,1]) \simeq \so_3$. Altogether
we get an injective morphism $\AA \to \C \times \so_3(\C)$. Since the action
of $\AA$ on $[3,1]$ is irreducible, the image of $\AA$ inside $\so_3(\C)$
is semisimple hence has dimension at least $3$, whence is $\so_3(\C)$
and $\AA' \simeq \so_3(\C)$. Since the action of $\AA$ on $[2,2]^+$
is non-trivial we have $Z(\AA) \neq \{ 0 \}$
hence the morphism $\AA \to \C \times \so_3(\C)$ is an isomorphism.

\subsubsection{$W = G(2,2,4)$} This is a Coxeter group of type $D_4$.
The irreducible representation of the Coxeter group of type $B_4$ are labelled by couples $(\la,\mu)$ of partitions
of total size $4$, hence the representations of $W$ inherit from Clifford theory
the labels $\{ \la, \mu \}$ for $\la \neq \mu$, and $\{ 2 \}^+$, $\{ 2 \}^-$, $\{ 11 \}^+$,$ \{ 11 \}^-$.
When $\mu = \emptyset$, the representation factors through the parabolic subgroup $\mathfrak{S}_4 = G(1,1,4)$ of type $A_3$,
hence these cases have already been treated. The remaining representations are,
up to tensorization by the sign character, $\{3,1\}$, $\{ 2 \}^+$, $\{ 2 \}^-$, $\{ 2 ,11 \}$, $\{ 21,1 \}$.
The last two split, when restricted to $\OO$, as $\{ 21,1 \}' + \{ 21,1 \}''$ and $\{ 2 ,11 \}' + \{ 2 ,11 \}''$,
thus affording $7$ irreducible representations of $\OO$ to care of. Using character theory we
get that these $7$ representations are selfdual except for $\{ 21,1 \}' $ and $ \{ 21,1 \}''$ ;
since $\{ 21,1 \}$ is selfdual this implies $\{ 21,1\} '' \simeq (\{ 21,1 \}')^*$.
The image of $\AA$ inside (the endomorphism algebras corresponding to) $\{3, 1 \}$, 
 $\{ 2 \}^+$, $\{ 2 \}^-$, $\{ 2 ,11 \}' \oplus \{ 2 ,11 \}''$ is thus $\so_4$, $\so_3$, $\so_3$, $\so_3 \times \so_3$, respectively.
 Using matrix models for these representations, we can compute the dimension of the image, and we get equality.
 Recall that $\so_3 \simeq \sl_2$ and that $\so_4 \simeq \so_3 \times \so_3$.
 
 Since $\{ 21,1 \}'' = (\{ 21,1 \}')^*$, the image of $\AA$ inside $\{ 21, 1 \} = \{ 21, 1 \}' + \{ 21, 1 \}''$ is isomorphic
 to the image inside $\{ 21, 1 \}'$, which is included in $\gl_4$. By direct computation we get that this image
 has
 dimension 16, hence it is all $\gl_4$. From all this one gets 6 simple ideals : $\so_3^{(0)}$ from the $\mathfrak{S}_4$-representation
 $\{ 31,\emptyset \}$, $\so_3^{(1)}, \so_3^{(2)}$ from $\{ 2 \}^+$, $\{ 2 \}^-$ , $\so_3^{(3)}$ from $\{ 2, 11 \}'$,
 $\so_3^{(4)}$ from $\{ 2,11 \}''$ and $\sl_4$ from $\{21, 1 \}'$. The $\so_3$ ideals are all distinct
 because $\so_3$ admits exactly one 3-dimensional irreducible representation. Thus $\AA'$
 contains $\sl_4 \times (\so_3)^5$, of dimension $30$. Direct computation shows $\dim \AA = 31$
 hence $\AA \simeq \C \times \sl_4 \times \so_4 \times (\so_3)^3 \simeq \C \times \sl_4 \times (\so_3)^5$.

\section{Rotation algebras : structure theorem in type $A$}
\label{sectrotA} 

In this section we let $\OO = \mathfrak{A}_n \subset \mathfrak{S}_n = W$ and
$\AA = \AA_n \subset \mathcal{H}_n = \mathcal{H}$ the associated rotation algebra. Recall that $\Irr(\mathfrak{S}_n)$
is parametrized by partitions $\la \vdash n$ and that, identifying  $\la \vdash n$
with the corresponding partition, one has $\la \otimes \eps = \la'$ the transposed
partition. By Clifford theory, irreducible representations of $\mathfrak{A}_n$
are thus parametrized by $\la \vdash n$ with $\la \neq \la'$ up
to identification of $\la$ and $\la'$, or in a more convenient way
with partitions $\la \vdash n$ with $\la > \la'$ for some arbitrarily chosen total
order on partitions (e.g. lexicographic ordering),
and
by $\la^{\pm}$ for $\la \vdash n$ with $\la = \la'$. Recall that a \emph{hook}
is a partition of the form $[n-k,1^k]$.
We introduce the following sets  : $\Lambda_ n = \{ \la \vdash n \ | \la > \la' \ \mbox{ and $\la$ not a hook} \}$,
 $S_n = \{ \la \vdash n \ | \la = \la' \ \mbox{ and $\la$ not a hook} \}$,
 $S_n^+ = \{ \la \in S_n \ | \ \la^{\pm } \mbox{\ have real type} \}$,
 $S_n^- = S_n \setminus S_n^+$.
 
 When $\la = \la'$, it is known that $\la^+$ and $\la^-$ have real type iff $(n-b(\la))/2$ is even,
 where $b(\la)$ is the length of the diagonal in the Young diagram $\la$ (see \cite{LIETRANSP}, lemme 5 and lemme 6).

\begin{theor} \label{theorotationA}
For $n \geq 5$,
$$
\mathcal{A}_n \simeq \so_{n-1} \oplus \left( \bigoplus_{\la \in \Lambda_n} \so(V_{\la}) \right)
\oplus \left( \bigoplus_{\la \in S_n^+} \so(V_{\la^+})\oplus\so(V_{\la^-}) \right)
\oplus \left( \bigoplus_{\la \in S_n^-} \sl(V_{\la^+}) \right)
$$
\end{theor}

As a corollary, the dimension of $\mathcal{A}_n$ for $n \geq 5$ is $16,112,1002,9115,86949,892531,9924091,\dots$.
For $n \geq 5$, we know by lemma \ref{lemsingleconjclass} and proposition \ref{propsemisimpA} that $\mathcal{A}_n$ is semisimple. Since $\mathcal{A}_n \subset \mathcal{H}_n$,
the decomposition of $\HG2^0$ obtained in theorem \ref{theodecHG2cox}  provides
an embedding of $\mathcal{A}_n$ inside the RHS of the equation.
We need to prove that it is surjective.
The case of $n = 5,6$ can be done by a simple computer
computation (one only needs to compare the dimension of both sides).
The case $n = 7$ is already large enough so that computer calculations need to be done modulo some prime $p$. The
RHS of theorem \ref{theorotationA} has dimension $1002$. Letting $V = V^{(1)} \subset \kk \mathfrak{A}_7$ denote the $\kk$-linear span of the $su-us$
for $s,u \in \mathcal{R}$, $V^{(i+1)} = V^{(i)} + [V,V^{(i)}]$, we used a C program and encoding of each entry inside one byte to check
on the regular representation of $\mathfrak{A}_7$ that, when $\kk = \mathbbm{F}_{113}$, $\dim V = 35$, $\dim V^{(2)} = 161$, $\dim V^{(3)} = 533$,
$\dim V^{(4)} = 987$, $\dim V^{(5)} = 1002$. Since we know that, for $\kk = \Q$, $\dim V^{(5)} \leq 1002$,
a straightforward application of Nakayama's lemma yields $\dim V^{(5)} = 1002$ for $\kk = \Q$, and this settles the case $n = 7$.
We thus assume $n \geq 8$ and start a proof by
induction (more precisely, we assume by induction that the natural map explicited above is an isomorphism for $n-1$).

We will use the following lemma. Here and in the sequel, $\rk \g$ denotes the semisimple rank of the semisimple Lie algebra $\g$.

\begin{lemma} \label{lemlie} Let $U$ be a $N$-dimensional $\C$-vector space endowed with a nondegenerate quadratic form,
$\h \subset \g \subset \so(U) = \so_N$
two semisimple Lie algebras such that $\g$ acts irreducibly on $U$, the action of $\h$ on $U$ is multiplicity-free,
and $\rk \h > N/4$. Then $\g$ is simple and, if $\rk \h \geq 5$ or $\dim N > 8$, then $\g = \so_N$.
\end{lemma} 
\begin{proof}
We have $\rk \g \leq \rk \so_N \leq N/2$ hence $\rk \h > N/4$ implies $\rk \h > (\rk \g)/2$.
This implies that $\g$ is simple (\cite{IH2}, lemma 3.2). Then
$2 \rk \g \leq 2(N/2) = \dim U < 4 \rk \h \leq 4 \rk \g$, hence we can apply the classification of \cite{IH2} lemma 3.4.
The exceptions are ruled out, as the selfdual ones appear only when $\g$ has rank at most 4 and $U$ has dimension at
most $8$, and $U$ is a selfdual representation of $\g$ as $\g \subset \so(U)$.
\end{proof}

We will also need the following combinatorial lemma.

\begin{lemma} \label{lemcombidimA}
\begin{enumerate}
\item Let $\la \vdash n$ with $n \geq 5$. Then $\dim \la \geq n-1$.
item Let $\la \vdash n$ with $n \geq 7$ and $\la = \la'$. Then either $n=7$, $\la$ is the hook $[4,1^3]$ with $\dim \la^{\pm} = 10$,
or $\dim \la^{\pm} \geq 21$.
\end{enumerate}
\end{lemma}

\begin{proof}
(1) is classical and easily deduced from Young's rule (restriction to $\mathfrak{S}_{n-1}$) by induction.
For (2), 
we first consider the case where $\la$ is a hook, $\la = [1+m,1^m]$ hence $n = 2m +1$, and
then $\dim \la = (2m)!/(m!)^2 \geq (2m-1) 2^m$. E.g. by Young's rule this dimension grows in function of
$m$, hence $\dim \la^{\pm} \geq 10$ for $n \geq 6$. If $\la$ is not a hook and $\la = \la'$ with $n \geq 7$,
then the diagram of $\la$ contains either the diagram $[4,2,1,1]$, which corresponds to a representation
of $\mathfrak{S}_7$ of dimension $70$ or
the diagram $[3,3,2]$, which corresponds to a representation
of $\mathfrak{S}_8$ of dimension $42$. It follows that $\dim \la^{\pm} \geq \frac{1}{2} \min(70,42) = 21$.

\end{proof}

For $\la \vdash n$, we denote $\rho_{\la} : \mathcal{A}_n \to \gl(V_{\la})$ and similarly,
when $\la = \la'$, $\rho_{\la^{\pm}} : \mathcal{A}_n \to \gl( V_{\la^{\pm}})$.
We claim that we only need to prove
$$
\left\lbrace \begin{array}{lclr} 
\rho_{[n-1,1]}(\mathcal{A}_n) &=& \so(V_{[n-1,1]}) \\
\rho_{\la}(\mathcal{A}_n) &=& \so(V_{\la}) & \mbox{ \ if \ } \la \in \Lambda_n \\
\rho_{\la^{\pm}}(\mathcal{A}_n) &=& \so(V_{\la^{\pm}}) & \mbox{ \ if \ } \la \in S_n^{+} \\
\rho_{\la^{+}}(\mathcal{A}_n) &=& \sl(V_{\la^+}) & \mbox{ \ if \ } \la \in S_n^{-} \\

\end{array} \right.
$$
Indeed, if this is the case, we get a collection of simple ideals of $\mathcal{A}_n$ from each representation in the
above list, by taking the orthogonal of their kernel with respect to the Killing form of $\mathcal{A}$. These ideals are
indeed simple, because the exceptional non-simple case $\so_4$ does not occur by our
assumption on $n$ ($n \geq 8$) by lemma \ref{lemcombidimA}.
If they are distinct, then the theorem is proved for this value of $n$, since $\mathcal{A}$
being semisimple is isomorphic to the direct product of its simple Lie ideals. Finally, they are indeed distinct,
as otherwise there would be 2 non-isomorphic representations among the above factoring through the same standard
Lie algebra of type $\so_N$ or $\sl_N$. But this is not possible, as the exceptional isomorphism $\so_6 \simeq \sl_4$
and the exceptional case of $\so_8$ (which has 3 non-isomorphic `standard' representations) are ruled out for $n \geq 8$ by lemma 
\ref{lemcombidimA}.

In order to prove the above equalities, we will apply repeatedly lemma \ref{lemlie} to the following situation :
$\g = \rho(\mathcal{A}_n)$, $\h = \rho(\mathcal{A}_{n-1})$. Since the restriction to $\mathfrak{A}_{n-1}$
of an irreducible representation of $\mathfrak{A}_n$ is multiplicity free (and since the restriction of a non-hook contains at most one hook)
so is the restriction to $\h$ of $\rho$ by proposition \ref{propirrA}.

We subdivide our analysis into several cases, and use the notation $\mu \nearrow \la$ as an abreviation for saying
that $\mu \vdash n-1$ and that the diagram
for $\mu$ is contained in the diagram of $\la$ (or : $\mu$ is deduced from $\la$ by removing one box).
We can assume that $\la$ is not a hook, except for the case $\la = [n-1,1]$. We start by doing
this case separately. Up to some additional computer calculation, we can assume $n \geq 12$. Then $\h \simeq \so_{n-2}$ has rank $\lfloor \frac{n-2}{2} \rfloor \geq \frac{n-2}{2} -1 > (\dim \la)/4 = (n-1)/4$, and $\g$ is simple by lemma \ref{lemlie},
and moreover $\rk h \geq 5$, $\dim \la > 8$, hence $\g = \so(V_{\la})$. From now
on, we can thus assume that $\la$ is \emph{not} a hook.

\subsection{Case $\mu \nearrow \la \Rightarrow \mu \mbox{ \ is not a hook }$}

Since $n \geq 7$, $\mu \nearrow \la \Rightarrow \dim \mu \geq 5$, and $\mu = \mu' \Rightarrow \dim \mu^{\pm}  \geq 8$.

\subsubsection{Case $\la \neq \la'$, and $\mu \nearrow \la \Rightarrow \mu \neq \mu'$}
\label{refsubsect91}
Then $\mu \nearrow \la \Rightarrow \mu' \neqnear \la$. Let $\Res_{\mathfrak{A}_{n-1}} \la = \mu_1 + \dots + \mu_d$.
Since $\dim \mu_i \geq 5$ we have $\dim \la \geq 5d$. By the induction assumption
$$
\rk \h = \sum_i \lfloor \frac{\dim \mu_i}{2} \rfloor \geqslant \sum_i \frac{\dim \mu_i}{2} - 1 \geq \frac{\dim \la}{2} - d \geq (\frac{1}{2} - \frac{1}{5} )\dim \la > \frac{\dim \la}{4}
$$
hence $\g$ is simple by lemma \ref{lemlie}, and moreover $\g = \so(V_{\la})$ unless possibly if $\dim \la \leq 8$.
But this implies $d=1$ and then $\h = \so(V_{\mu_1}) = \so(V_{\la})$.

\subsubsection{Case $\la =\la'$, $\la^{\pm}$ of real type and $\mu \nearrow \la \Rightarrow \mu \neq \mu'$}
\label{refsubsec92}
Then, letting $\Res_{\mathfrak{A}_{n-1}} \la^+ = \mu_1 + \dots + \mu_d$, $\mu_i \vdash n-1$, we have 
$$
\rk \h = \sum_{\mu \nearrow \la} \frac{1}{2} \lfloor \frac{\dim \mu}{2} \rfloor \geqslant \frac{1}{2} \sum_{\mu \nearrow \la}
\frac{\dim \mu}{2} - 1 \geqslant \frac{1}{4} \dim \la - d \geqslant \frac{1}{2} \dim \la^{\pm} - d.
$$
Again $\dim \la \geq 5d$ hence $\rk \h > (\dim \la)/4$ and $\g$ is simple and $\g = \so(V_{\la^{\pm}})$
by lemma \ref{lemlie}, as $\dim \la^{\pm} > 8$ by lemma \ref{lemcombidimA}.

\subsubsection{Case $\la = \la'$, $\la^{\pm}$ of real type and $\exists \mu_0 \nearrow \la \ | \ \mu'_0 = \mu_0$}

\label{refsubsec93}

We recall that we denote $b(\la)$ the length of the diagonal in the Young diagram $\la$.
In our situation, $\mu_0$ is uniquely determined by the condition $\mu'_0 = \mu_0$, and
we have $b(\mu_0) = b(\la)-1$, and $n-1 - b(\mu_0 = n- b(\la)$ hence the $\mu_0^{\pm}$
have real type. Thus
$$
\rk h = \left( \sum_{\stackrel{\mu \nearrow \la}{\mu \neq \mu'}} \frac{1}{2} \lfloor \frac{\dim \mu}{2} \rfloor \right)
+ \lfloor \frac{\dim \mu_0^{\pm}}{2} \rfloor \geqslant 
\frac{1}{2} \left(\sum_{\stackrel{\mu \nearrow \la}{\mu \neq \mu'}}  \frac{\dim \mu}{2}-1 \right) + \frac{\dim \mu_0^{\pm}}{2} -1.
$$
Letting $\delta(\la) = \# \{ i \ | \ \la_i \neq \la_{i+1} \} = \# \{ \mu \ | \ \mu \nearrow \la \}$, we get
$$
\rk \h \geqslant \frac{1}{2} ( \frac{1}{2} (\dim \la - \dim \mu_0) - \delta(\la) + 1) + \frac{\dim \mu_0^{\pm}}{2} -1
\geqslant \frac{1}{2} \dim \la^{\pm} - \frac{1}{2} (\delta(\la) -1) > \frac{\dim \la^{\pm}}{4}
$$ 
if and only if $\dim \la^{\pm} > 2 (\delta(\la) + 1)$. On the other hand, $\Res_{\mathfrak{A}_{n-1}} \la^{\pm}
= \mu_0^{\pm} + \mu_1 + \dots + \mu_{d-1}$ and $\delta(\la) = 1 + 2 (d-1) = 2d-1$,
hence $2 (\delta(\la) + 1) = 4d < \dim \la^{\pm}$ because $\dim \mu_0^{\pm} \geq 5$ and $\dim \mu_i \geq 5$
for $i \geq 1$. It then follows from lemma \ref{lemlie} that $\g$ is simple and $\g = \so(V_{\la^{\pm}})$.

\subsubsection{Case $\la = \la'$, $\la^{\pm}$ of complex type and $\mu \nearrow \la \Rightarrow \mu \neq \mu'$}

\label{refsubsect94}

As in subsection \ref{refsubsec92} we get $\rk \h > (\dim \la^{\pm})/4$. If $\rk \g > (\dim \la^{\pm})/2$
then $\g = \sl(V_{\la^{\pm}})$ (\cite{IH2}, lemma 3.1). Otherwise $\rk \g \leq (\dim \la^{\pm})/2 < 2 \rk \g$
and $\g$ is simple, thanks to \cite{IH2}, lemma 3.2. Thus $\rk \g \geq \rk \h > (\dim \la^{\pm})/4 > (\dim \la^{\pm})/5$
hence by \cite{IH2} lemma 3.5 this implies $\g = \sl(V_{\la^{\pm}})$ when $\rk \g \geq 10$, that
is $\rk \g > 9$, which holds true as soon as $\dim \la^{\pm} > 36$, i.e. $\dim \la > 72$. Moreover,
for $n \geq 7$, $\la$ contains either $[4,2,1,1]$, or $[3,3,2]$, which have dimension $70$ and $42$, respectively.
Since $\la = \la'$ and $\la$ is not a hook, this implies that either $\la = [3,3,2]$ or $\la$ contains $\mu \neq \mu'$ with $\mu \supset [3,3,2]$
hence $\dim \la \geq \dim \mu + \dim \mu' \geq 84$, in which case we are done. If $\la = [3,3,2]$,
then $\h$ contains $\so(V_{[3,3,1]})$ which has rank 10, and then $\g = \sl(V_{\la^{\pm}})$ by lemma \ref{lemlie}
in this case, too.

\subsubsection{Case $\la = \la'$, $\la^{\pm}$ of complex type and $\exists \mu_0 \nearrow \la \ | \ \mu'_0 = \mu_0$}
As in part \ref{refsubsec93}, the $\mu_0^{\pm}$ have the same type as $\la^{\pm}$, so they
have complex type here. In the same way we get $\rk \h > (\dim \la^{\pm})/4$ and as in
part \ref{refsubsect94} we deduce from this that $\g$ is simple, and $\g = \sl(V_{\la^{\pm}})$ as soon
as $\rk \g \geq 10$, which holds true as soon as $\dim \la^{\pm} > 36$, except possibly when $\la = [4,2,1,1]$,
which is ruled out because it is a hook, or $\la = [3,3,2]$, which is out of the scope of this case
(and has been dealt with earlier).

\subsubsection{Case $\la \neq \la'$, and $\exists \mu_0 \nearrow \la \ | \ \mu'_0 = \mu_0$}
Again, such a $\mu_0^{\pm}$ is then uniquely determined, and we can write
$\Res_{\mathfrak{A}_{n-1}} \la = \mu_0^+ + \mu_0^- + \mu_1 + \dots + \mu_{d-2}$. We subdivide in two
subcases. Either the $\mu_0^{\pm}$ have real type, in which case
$$
\rk \h= \lfloor \frac{\dim \mu_0^+}{2} \rfloor + \lfloor \frac{\dim \mu_0^-}{2} \rfloor + \sum_{i=1}^{d-2} 
\lfloor \frac{\dim \mu_i}{2} \rfloor \geqslant \frac{\dim \la}{2} -2
$$
and we conclude as in \ref{refsubsect91}, or the $\mu_0^{\pm}$ have complex type,
$$
\rk \h = \dim \mu_0^+ - 1 +  \sum_{i=1}^{d-2} 
\lfloor \frac{\dim \mu_i}{2} \rfloor \geqslant \frac{\dim \la}{2} - (d-1) \geqslant \frac{\dim \la}{2} - \frac{\dim \la}{5} + 1 > \frac{\dim \la}{4} 
$$
and we conclude again as in \ref{refsubsect91}.

\subsection{$\exists \mu \nearrow \la$, $\mu$ is a hook}

As in \cite{LIETRANSP} we introduce the partitions/diagrams $D(a,b) = [a+2,2,1^b] \vdash n = a+b+4$.
Since $n \geq 7$ we can assume $a+b \geq 3$. From \cite{LIETRANSP} we recall that
$\dim D(a,b) = \frac{b+1}{a+2} \left( \begin{array}{c} a+b+2 \\ a \end{array} \right) (a+b+4)$.
Clearly $D(a,b)' = D(b,a)$, and our assumption means that $\la = D(a,b)$ for some $a,b$.

\subsubsection{Case $\la = [n-2,2]$ (i.e. $a=0$ or $b= 0$)}

$\Res_{\mathfrak{A}_{n-1}} \la = [n-2,1] + [n-3,2]$ and
$$
\rk \h = \lfloor \frac{\dim [n-2,1]}{2} \rfloor +  \lfloor \frac{\dim [n-3,2]}{2}  \rfloor \geqslant \frac{\dim \la}{2} -2 > \frac{\dim \la}{4}
$$
if and only if $(\dim\la)/4 > 2$, that is $\dim \la > 8$, which is true for $n \geq 6$. We conclude by lemma \ref{lemlie}.

\subsubsection{Case $a,b \geq 1$, $a\neq b$}
\label{refsubsect102}
$\Res_{\mathfrak{A}_{n-1}} \la = \Res_{\mathfrak{A}_{n-1}} D(a,b) = D(a-1,b) + D(a,b-1) + [a+2,1^{b+1}]$.
We first assume $|a-b| > 1$. Then
$$
\rk \h = \lfloor \frac{\dim D(a-1,b)}{2} \rfloor + \lfloor \frac{\dim D(a,b-1)}{2} \rfloor + \lfloor \frac{n-2}{2} \rfloor
\geqslant \frac{\dim D(a,b) - \dim [a+2,1^{b+1}]}{2} - 2 + \lfloor \frac{n-2}{2} \rfloor
$$
hence
$$
\rk \h \geqslant \frac{\dim \la}{2} - \frac{1}{2} \dim [a+2,1^{b+1}] + \lfloor \frac{n-2}{2} \rfloor - 2.
$$ 
Since $n \geq 8$ we have
$\lfloor \frac{n-2}{2} \rfloor \geq 3$ hence
$$
\rk \h > \frac{\dim \la}{2} - \frac{1}{2} \dim [a+2, 1^{b+1}] \geqslant \frac{\dim \la}{4}
$$
as soon as $(\dim \la)/4 \geq (\dim [a+2,1^{b+1}]) /2$, that is $\dim D(a,b) \geq 2 \dim [a+2,1^{b+1}]$.
Up to exchanging $a,b$, we can assume $a \geq 3$ , $b \geq 1$. An easy calculation
shows $D(a,b) /2 \dim [a+2,1^{b+1}] \geq \frac{64}{14} > 2$. Lemma \ref{lemlie} then
implies $\g = \so(V_{\la})$ as soon as $\dim \la > 8$, and we know $D(a,b) \supset D(3,1) = [5,2,1]$
which has dimension $64$.

\subsubsection{Case $a=b$, $a,b \geq 1$}

That is, $\la = D(a,a)$ with $a \geq 1$. Then
$\Res_{\mathfrak{S}_{n-1}} \la = D(a-1,a) + D(a,a-1) + [a+2,1^{a+1}]$ and
$\Res_{\mathfrak{A}_{n-1}} \la^{\pm} = D(a-1,a) + [a+2,1^{a+1}]^{\pm}$.
There are two subcases. First assume that $\la^{\pm}$ has real type ;
this means that $(n-2)/2$ is even, that is $a$ is odd. The case $a = 1$ is $[3,2,1]$
which is ruled out by $n \geq 8$, hence $a \geq 2$. Then
$$
\rk \h = \lfloor \frac{\dim D(a-1,a)}{2} \rfloor +\lfloor \frac{n-2}{2} \rfloor \geqslant \frac{\dim D(a-1,a)}{2} -1 +\lfloor \frac{n-2}{2} \rfloor 
$$

and
$$
\rk \h \geqslant \frac{\dim D(a,a)^{\pm}}{2} - \frac{\dim [a+2,1^{a+1}]}{2} - 1 + \lfloor \frac{n-2}{2} \rfloor 
> \frac{\dim \la^{\pm}}{4} 
$$
as soon as $\dim D(a,a) / \dim [a+2,1^{a+1}] \geq 2$. From the explicit formulas for the dimensions
this is equivalent to $(1 - \frac{1}{a+2})^2(2a+4) \geq 2$, which is true. Hence $\g$ is simple and
$\g = \so(V_{\la^{\pm}})$ as soon as $\dim \la^{\pm } > 8$, which is true since $n \geq 8$.

\medskip

We now assume that $\la^{\pm}$ has complex type, that is $a$ is even. Since $n = 2a +1 \geq 8$,
this implies $a \geq 4$. As before we get $\rk \g > (\dim \la^{\pm})/4 > (\dim \la^{\pm})>5$, hence
either $\rk \g > (\dim \la^{\pm})/2$ and we get immediately $\g = \sl(V_{\la^{\pm }})$ by
\cite{IH2} lemma  5.1, or $\rk \g < 2 \rk \h$ and $\g$ is simple by \cite{IH2} lemma 3.2. Then
$\rk \g > \frac{\dim \la^{\pm}}{5}$ implies $\g = \sl(V_{\la^{\pm}})$ by \cite{IH2} lemma 3.5 as soon
as $\rk \g > 9$, which is the case as soon as $\dim \la^{\pm} > 36$, that is $\dim \la > 72$. Since
$\dim \la \geq \dim D(3,3) = 448$ this concludes this case.

\subsubsection{Case $|a-b| = 1$, $a,b \geq 1$}

We can assume $a= b+ 1 \geq 2$, that is $\la = D(b+1,b)$. Then $\Res_{\mathfrak{A}_{n-1}}
D(b+1,b) = D(b,b)^+ + D(b,b)^- + D(b+1,b-1) + [b+3,1^{b+1}]$. Moreover,
$D(b,b)^{\pm}$ has real type if $b$ is odd, and complex type if $b$ is even.

We first assume that $b$ is odd.
Since $n \geq 8$ this implies $b \geq 3$, hence $n \geq 11$.
Then
$$
\rk \h \geqslant \frac{\dim D(b,b)^+}{2} + \frac{\dim D(b,b)^-}{2} + \frac{\dim D(b+1,b-1)^+}{2}  - 3 + \lfloor \frac{n-2}{2} \rfloor
\geqslant \frac{\dim \la}{2} - \frac{\dim [b+3,1^{b+1}]  }{2}
$$
because $- 3 + \lfloor \frac{n-2}{2} \rfloor \geq 0$ for $n \geq 8$. Thus $\rk \h > (\dim \la)/4$ as soon
as $(\dim D(b+1,b])/(\dim [b+3,1^{b+1}]) > 2$, and this is a consequence of the computation in \ref{refsubsect102}.
Hence $\g$ is simple and $\g = \so(V_{\la})$ as soon as $\dim \la > 8$, which is true as $\dim D(2,1) = 35$.

We now assume that $b$ is even, hence $b \geq 2$. Then
$$
\rk \h \geq \dim D(b,b)^+ -1 + \frac{\dim D(b+1,b-1)}{2} -1 + \lfloor \frac{n-2}{2} \rfloor > \frac{\dim \la}{2} - \frac{\dim [b+3,1^{b+1}]}{2}
\geqslant \frac{\dim \la}{4}
$$
and we conclude as before.

\section{Rotation algebras : structure theorem in types $D,E$}
\label{sectrotDE}
\def\LRef{\Lambda \mathrm{Ref}}
\def\Ref{\mathrm{Ref}}

We use the notations of \cite{IH2} for the representations of $W$. In particular,
$\LRef = \LRef(W)$ denotes the irreducible representations
deduced from a reflection representation by  taking some alternating power of it
and tensoring by a linear character. When $W = D_n$ We denote $\Ind\LRef$
the set of irreducible representations whose restriction to $D_{n-1}$
have a component in $\LRef(W_0)$, for $W_0$ the standard parabolic subgroup
of type $D_{n-1}$, $\Irr''(W) = \Irr(W) \setminus \Ind \LRef$,
and $\Irr'(W) = \Irr(W) \setminus \LRef$. Recall from
\cite{IH2} that $\Ind \LRef \supset \LRef$, hence $\Irr''(W) \subset \Irr'(W) \subset \Irr(W)$.
In case $W$ has type $E_6,E_7,E_8$, we take for $W_0$ a standard parabolic subgroup of type $D_5,E_6,E_7$.
We let $\OO$ and $\OO_0$ denote the rotation subgroups of $W$ and $W_0$, respectively. We let $\Irr''(\OO)$ denote the
set of the $\rho \in \Irr(\OO)$ which are constituents of the restriction of some element of $\Irr''(\OO)$, and
we define similarly $\Irr'(\OO)$, $\Irr'(\OO_0)$.

Here reflection representation means that
the reflections of $W$ act by reflections in the representations, but such
a representation is not assumed to be faithful. 
We denote $\Ref$ the set of reflections representations.
In type $D_n$, they
are labelled $\{ [n-1], [1] \}$ ($n$ dimensional) and $\{ [n-1,1],\emptyset \}$ ($n-1$ dimensional).
Recall that the representations $\{ \la, \emptyset \}$ factor through $\mathfrak{S}_n$.
We parametrize the representations of $\OO$ as follows.

\begin{enumerate}
\item couples $(\la,\mu)$ for $\la > \mu,\mu' $ and $\la \geq \la'$, which originate from representations $\{ \la, \mu \}$ of $W$ with $\la \neq
\mu$, $\la \neq \mu'$ and $(\la,\mu) \neq (\la',\mu')$. They have real type.
\item signed couples $(\la,\mu)^{\pm}$ for $\la > \mu$, $\la = \la'$, $\mu = \mu'$, which originate from representations $\{ \la, \mu \}$ of $W$. 

\item signed partitions $(\la)^{\pm}$ with $\la \vdash n/2$ for $\la > \la'$, which originate from representations $\{ \la, \la' \}$ of $W$
with $\la \neq \la'$. 
\item signed partitions $((\la))^{\pm}$ with $\la \vdash n/2$ for $\la > \la'$, which originate from representations $\{ \la \}^{\pm}$ of $W$
with $\la \neq \la'$.
\item partitions $((\la))$ with $\la \vdash n/2$,  which originate from representations $\{ \la \}^{\pm}$ of $W$
with $\la = \la'$ and $n/2$ odd.
\item doubly signed partitions $((\la))^{\pm \pm}$ with $\la \vdash n/2$,  which originate from representations $\{ \la \}^{\pm}$ of $W$
with $\la = \la'$ and $n/2$ even : $\Res_{\OO} \{ \la \}^{\pm} = \{ \la \}^{\pm +} \oplus \{ \la \}^{\pm -}$.
\end{enumerate}

The reason for the existence of the last 2 cases is the following lemma.

\begin{lemma} \label{clifsubtletypeD} Let $\la \vdash m$ with $\la = \la'$. Then $\{ \la \}^{\pm} \otimes \eps = \{ \la \}^{\pm}$ if $m$ is even,
$\{ \la \}^{\pm} \otimes \eps = \{ \la \}^{\mp}$ if $m$ is odd.
\end{lemma}
\begin{proof} 
Let $C = \{ \pm 1 \}^{2m}$, $H < C$ the index two subgroup made of the $2m$-tuples of product 1. Then
$\tilde{W} = \mathfrak{S}_{2m} \ltimes C$ and $W = \mathfrak{S}_{2m} \ltimes H$ are the Coxeter
groups of type $B_n$ and $D_n$, respectively, and $A = H \rtimes \mathfrak{A}_n$.
The representation $(\la, \la)$ of $B_n$ can be obtained as $\Ind_{\mathfrak{S}_m^2 \ltimes C}^{\mathfrak{S}_n \ltimes C} \tilde{\om} (\la \boxtimes \la)$
where $\la \boxtimes \la$ denotes the exterior Kronecker product of two copies of the
representation $\la$ of $\mathfrak{S}_m$, extended trivially to $\mathfrak{S}_m^2 \ltimes C$,
and $\tilde{\om}$ the trivial extension of $\om : C \to \{ \pm 1 \}$ to
$\mathfrak{S}_m^2 \ltimes C$, where $\om(x_1,\dots,x_{2m}) = x_{m+1} \dots x_{2m}$
(note that this character is stabilized by $\mathfrak{S}_m^2$, hence that this `trivial
extension' makes sense). We need to prove that, in case $\la = \la'$, the
$\Res_{H \rtimes \mathfrak{A}_n}^{C \rtimes \mathfrak{S}_n} \Ind_{\mathfrak{S}_m^2 \ltimes C}^{\mathfrak{S}_n \ltimes C} \tilde{\om} (\la \boxtimes \la)$ has 2 irreducible constituents
if $m$ is odd, and 4 otherwise. Note that, by elementary Clifford theory, the number of irreducible constituents is necessary either 2 or 4.

By Mackey formula, and because $  H \rtimes \mathfrak{A}_n \backslash \mathfrak{S}_n \ltimes C / \mathfrak{S}_m^2 \ltimes C = \{ 1 \}$, we get that this restriction is
$
\Ind_{H \rtimes (\mathfrak{A}_n \cap \mathfrak{S}_m^2)}^{H \rtimes \mathfrak{A}_n}
\Res_{H \rtimes(\mathfrak{A}_n \cap \mathfrak{S}_m ^2)}^{C \rtimes \mathfrak{S}_m^2}
\tilde{\omega} (\la \boxtimes \la)   
$. By elementary Clifford theory, since $\la' = \la$, we have that $\Res_{\mathfrak{A}_n \cap \mathfrak{S}_m^2}^{\mathfrak{S}_m^2}
\la \boxtimes \la$ is the sum of two irreducible constituents $(\la \boxtimes \la)^+$
and $(\la \boxtimes \la)^-$, thus $ \Res_{C \rtimes(\mathfrak{A}_n \cap \mathfrak{S}_m ^2)}^{C \rtimes \mathfrak{S}_m^2}
\tilde{\omega} (\la \boxtimes \la) = \tilde{\omega} (\la \boxtimes \la)^+
+ \tilde{\omega} (\la \boxtimes \la)^-$. Since the $\Res_{H \rtimes(\mathfrak{A}_n \cap \mathfrak{S}_m ^2)}^{C \rtimes(\mathfrak{A}_n \cap \mathfrak{S}_m ^2)}
\tilde{\omega} (\la \boxtimes \la)^{\pm}$ are clearly irreducible, we get that $\Res_{H \rtimes(\mathfrak{A}_n \cap \mathfrak{S}_m ^2)}^{C \rtimes \mathfrak{S}_m^2}
\tilde{\omega} (\la \boxtimes \la)$ has two irreducible constituents, namely the $\zeta_{\pm} = \tilde{\omega}_H (\la \boxtimes \la)^{\pm}$,
where $\omega_H$ denotes the restriction to $H$ of $\omega$, and $\tilde{\omega}_H$ its extension to $H \rtimes (\mathfrak{A}_n \cap \mathfrak{S}_m^2)$.
Let $\zeta$ denote one of these constituents.
One readily gets that $\{ g \in H \rtimes \mathfrak{A}_n \ | \zeta^g \simeq \zeta \}  = H \rtimes \{ g \in \mathfrak{A}_{2m} \ | \ \om_H^g \simeq \om_H \}
= H \rtimes  (\mathfrak{A}_n \cap\{ g \in \mathfrak{S}_{n} \ | \ \om_H^g \simeq \om_H \})$, and that
$\{ g \in \mathfrak{S}_n \ | \ \om_H^g \simeq \om_H \}$ is $\mathfrak{S}_m^2 \rtimes < \sigma >$ where $\sigma = (1 ,2m)(2,  2m-1)(3,  2m-2)\dots$.
If $m$ is odd, $\sigma \not\in \mathfrak{A}_n$, $I(\zeta) = \{ g \in H \rtimes \mathfrak{A}_n \ | \zeta^g \simeq \zeta \} = H \rtimes (\mathfrak{A}_n \cap {S}_m^2)$
and $\Ind_{H \rtimes (\mathfrak{A}_n \cap \mathfrak{S}_m^2)}^{H \rtimes \mathfrak{A}_n} \zeta$ is irreducible by Clifford's theorem (\cite{CURTISREINER} \S 11)
which proves the claim. 
If $m$ is even, then
$\Ind_{H \rtimes (\mathfrak{A}_n \cap \mathfrak{S}_m^2)}^{H \rtimes \mathfrak{A}_n} \zeta= \Ind_{I(\zeta)}^{H \rtimes \mathfrak{A}_n} \Ind_{H \rtimes (\mathfrak{A}_n \cap \mathfrak{S}_m^2)}^{I(\zeta)} \zeta$, $H \rtimes (\mathfrak{A}_n \cap \mathfrak{S}_m^2)$ has index 2 in $I(\zeta)$, hence by elementary
Clifford theory $\Ind_{H \rtimes (\mathfrak{A}_n \cap \mathfrak{S}_m^2)}^{I(\zeta)} \zeta$ has two irreducible constituents. Thus 
$\Ind_{H \rtimes (\mathfrak{A}_n \cap \mathfrak{S}_m^2)}^{H \rtimes \mathfrak{A}_n} \zeta$ has at least 2 irreducible constituents, which proves the claim.
\end{proof}

We notice for future use that elementary  applications of the branching rules show
the following.

\begin{lemma} In cases (3-6) above, the restriction to $\OO_0$ (rotation subgroup of type $D_{n-1}$) is multiplicity free and
made of representations of \emph{real} type.
\end{lemma}

\begin{theor} If $W$ is an irreducible Coxeter group of type ADE, then $\AA = \mathcal{L}_{1}(\OO) \cap \mathcal{H}'$,
unless $W$ has type $D_4$, $A_3$, $A_2$.
\end{theor}

According to the theorem, for $n \geq 5$, the dimension of $\AA$ in type $D_n$ is 
390, 5314, 78758, 1282059, 23189432, 464312278, 10217797426, 245243928461, 6376443304559 \dots
For $E_6,E_7,E_8$ it is 12593, 722403, 174117236.

One clearly has an embedding $\AA \into \mathcal{L}_{1}(\OO) \cap \mathcal{H}$,
and it falls into the semisimple part of $\mathcal{H}'$ because $\AA$ is semisimple for $n \geq 5$ (proposition \ref{propsemisimpA} and lemma \ref{lemsingleconjclass}).
We do a proof by induction, assuming the theorem proved for $W_0$.

\begin{lemma} \label{leminductreel} Assume that the theorem holds for $W_0$, that $\rho \in \Irr''(\OO)$ has real type and dimension $> 8$, that its
restriction to $\OO_0$ is multiplicity free, and that $\forall \varphi \in \Irr'(\OO_0) \ \dim \varphi \geq 5$. Then
$\rho(\mathcal{A}) = \so(V_{\rho})$.
\end{lemma}
\begin{proof} Letting $\g = \rho(\AA)$ and $\h = \rho(\AA_0)$,
we can write $\Res_{\OO_0} \rho = \rho_1 + \dots + \rho_r + \varphi_1 + \varphi_1^* + \dots + \varphi_s + \varphi_s^*$
with the $\rho_i , \varphi_i \in \Irr'(\OO_0)$, the $\rho_i$ having real type and the $\varphi_i$ complex type.
Then
$$
\rk \h \geqslant \sum_{i=1}^r \lfloor \frac{\dim \rho_i}{2} \rfloor + \sum_{i=1}^s (\dim \varphi_i - 1)
\geqslant\sum_{i=1}^r \left( \frac{\dim \rho_i}{2} -1 \right)+ 2 \left( \sum_{i=1}^s \frac{\dim \varphi_i}{2} \right) - s
$$ 
that is $\rk \h \geq \frac{\dim \rho}{2} - r - 2s$. Since we have $\dim \rho_i >4$ and $\dim \varphi_i > 4$,
then $\rk \h > \frac{\dim \rho}{4}$ and the conclusion follows from lemma \ref{lemlie}.

\end{proof}

\begin{lemma} \label{lemlie2} Assume that the theorem holds for $W_0$, that $\rho \in \Irr''(\OO)$ has complex type, and $\rk \h > \dim \rho/4$
with $\dim \rho > 21$. If the restriction of $\rho$ to $\OO_0$ is multiplicity free, then $\g = \sl(V_{\rho})$ .
\end{lemma}
\begin{proof}
$\g$ is simple by \cite{IH2} lemma 3.3 (I). Since $\rho$ is not selfdual and $\dim \rho > 1$, the conclusion
follows by \cite{IH2} lemma 3.4.
\end{proof}

The following lemma is similar to \cite{IH2}, lemma 2.25 : 

\begin{lemma} \label{lemnosl2} Assume $W$ has type $A_n, n \geq 5$, $D_n, n \geq 6$, $E_6$, $E_7$, $E_8$ and
$\rho \in \Irr(\OO)$ with $\dim \rho > 1$. Then $\rho(\mathcal{A}) \subset \gl(V_{\rho})$
do not contain any simple Lie ideal of rank $1$.
\end{lemma}
\begin{proof}
By lemma \ref{lempolordre3} we know that $\rho(\mathcal{A}) \subset \sl(V_{\rho}) = V_{\rho} \otimes V_{\rho^*}$
is invariant under $A$-conjugation, and that  so is any simple Lie ideal. As a consequence, a $\sl_2$ ideal would provide
a 3-dimensional sub-representation of $V_{\rho} \otimes V_{\rho^*}$. It is easily checked that, for the $W$ listed above, the smallest
non-linear irreducible character of $W$ has degree $5$, hence a 3-dimensional such a representation should be a sum of 1-dimensional ones. But for the same reason and by Clifford theory, there are only one $1$-dimensional character of $A$, and it can can appear only with multiplicity 1 in a tensor product of the form $V_{\rho} \otimes V_{\rho^*}$, according to Schur's lemma. This contradiction proves the lemma.
\end{proof}

We now concentrate on the case of type $D_n$, and do a proof by induction. We postpone the case of $D_5$
to subsection \ref{refsubsectexcept}, so we can assume $n \geq 6$.
The case of $W$ of type $A_n$ has been done before. We let $\Irr_0(W)$ the set of irreducible
representations of $W$ which do not factor through $A_{n-1}$, and $\Irr'_0 = \Irr' \cap \Irr_0$,
$\Irr''_0 = \Irr'' \cap \Irr_0$.

We let $\Irr_0''(\OO)$ (resp. $\Irr'_0(\OO)$) denote the set of irreducible representations of $\OO$ which appear inside
the restriction of an element of $\Irr_0''(W)$ (resp $\Irr'_0(W)$). We have the following combinatorial lemma.

\begin{lemma} {\ } \ \label{lemdimD}
\begin{enumerate}
\item If $\rho \in \Irr_0(\OO)$, in type $D_n$ for $n \geq 5$, has dimension $4,6$ or $8$, then
$\rho = ([n-1],[1])$.
\item If $\rho \in \Irr'_0(\OO)$ in type $D_n$ for $n \geq 5$, then $\dim \rho \geq 5$ (if $n \geq 6$ then $\dim \rho \geq 10$,
if $n \geq 7$ then $\dim \rho \geq 21$, if $n \geq 8$ then $\dim \rho \geq 48$).
\item If $\rho \in \Irr''_0(\OO)$ in type $D_n$, then $\dim \rho \geq 10$ when $n \geq 6$, and $\dim \rho \geq 35$ when $n \geq 7$. 
\item If $\rho \in \Irr'(\OO)$ in type $D_n$, then $\dim \rho \geq 5$ when $n \geq 5$.
\end{enumerate}
\end{lemma}
\begin{proof}
Under the assumptions of (1), by Clifford theory, $\rho$ appears in the restriction of some $\tilde{\rho} \in \Irr_0(W)$
with $\dim \rho \in \{ 4,6,8,12,16 \}$. For $n =7$, from the character table we get that all elements in $\Irr_0(W)$
have dimension at least $21$, except for  $\{ 6, 1 \}$ and $\{ 1^6, 1 \}$. If $n \geq 8$, then the restriction to
the obvious parabolic subgroup $H$ of type $D_7$ has to
contain $\{ 6, 1 \}$ or $\{ 1^6, 1 \}$ ; because of the branching rule, if it is not of the form $\{ n-1,1 \}$ or $\{ 1^{n-1},1 \}$ then its restriction to $H$ has to contain another constituent, thus of dimension at least 21. A check of the character tables shows that, for $5 \leq n \leq 7$,
there are no other choices as well.

\end{proof}

Note that $\Irr''_0(\OO) = \emptyset$ for $W$ of type $D_5$.
As in the case of type $A$, we are thus reduced to prove, when $n \geq 5$, that

$$
\left\lbrace \begin{array}{lcll}  \rho_{\{ [n-1],[1] \}}(\AA) &=& \so_n \\
\rho_{\{ [n-1],[1] \}}(\AA) &=& \so_n \\
\rho(\AA) & = & \so(V_{\rho})& \mbox{\ if\ } \rho\in \Irr'(A)  \mbox{ has real type }\\
\rho(\AA) & = & \sl(V_{\rho})& \mbox{\ if\ } \rho\in \Irr'(A)  \mbox{ has complex type }\\
\end{array} \right.
$$

\subsection{Case $\rho  = ([n-1],[1])$}

Since $\h$ contains $\so_{n-1}$, $\rk \g \geq \lfloor \frac{n-1}{2} \rfloor \geq (\dim \rho)/4 = n/4$ as
soon as $n > 6$. Moreover $\dim \rho > 8$ when $n > 8$, so we conclude by lemma \ref{lemlie},
the cases $n \leq 8$ being checked by computer.

\subsection{Case $\rho \in \Ind\LRef$}

\subsubsection{$\rho = ([n-p],[2,1^{p-2}])$}

If $p = 2$, $\Res \rho = ([n-3],[2]) + ([n-2],[1])$. Assuming $n \geq 6$ we have $n-3 > 2$.
By induction we have
$$
\rk \h = \lfloor \frac{\dim ([n-3],[2])}{2} \rfloor + \lfloor \frac{\dim ([n-2],[1])}{2} \rfloor \geqslant \frac{\dim \rho}{2}  > \frac{\dim \rho}{4}
$$
as soon as $\dim \rho > 8$, which holds true for $n \geq 5$. The case $n = 5$ is checked by computer.

If $p = n-1$, i.e. $\rho = ([1],[2,1^{n-3}])$, we assume again $ n \geq 6$. Then by induction
we similarly get
$$
\rk \h \geqslant \frac{\dim \rho}{2} - 3 - \frac{\dim ([1],[1^{n-2}])}{2} + \frac{n-1}{2} > \frac{\dim \rho}{4}
$$
as soon as $\dim \rho > 6$, which is true. The case $n = 5$ is checked by computer.

If $p \not\in \{ 2 , n-1 \}$, that is $n-p \geq 2$ and $p -2 \geq 1$. As before, we get
$\rk \g > (\dim \rho)/4$ as soon as $\dim \rho > 6 + 2 \dim ([n-p],[1^{p-1}]) - 2(n-1)$.
Now $\dim \rho = \dim \{ [n-p],[2,1^{p-2}] \}$ and the restriction rule
to $W(D_{n-2})$ shows that the restriction of $\{ [n-p],[2,1^{p-2}] \}$
contains 
the restriction of $\{ [n-p],[1^{p-1}] \}$. As a consequence, 
$\dim \rho \geq 2 \dim ([n-p],[1^{p-1}])$ hence $\rk \g > (\dim \rho)/4$ as soon as $n \geq 5$.

In all these cases, $\dim \rho > 8$ for $n \geq 5$, hence lemma \ref{lemlie} shows that $\rho(\AA) = \so(V)$.

\subsubsection{$\rho = ([1],[n-p,1^{p-1}])$}

We can assume $n-p \geq 2$, hence $n \geq p+2 \geq 4$,
and $p-1 \geq 1$, hence $ p \geq 2$.

We have  $\Res \rho = 
(\emptyset,[n-2,1]) + ([1],[n-p,1^{p-2}]) + ([1],[n-p-1,1^{p-1}])$ as a sum of irreducible constituent
unless one of the partitions $[n-p,1^{p-2}]$ and $[n-p-1,1^{p-1}]$ is selftransposed, but also in this case the restriction
is multiplicity free. Since $\rho$ has real type and all the constituents have dimension at least 5 for $n \geq 7$, the same proof as in lemma \ref{leminductreel} yields the conclusion for $n \geq 7$. The case $n = 6$ is dealt with in section \ref{refsubsectexcept}.

\subsection{$\rho \in \Irr''_0(\OO)$, cases (3)-(6)}
\label{refsubsectcas36}
In that case, $\rho$ has for restriction $\rho_1+ \dots + \rho_r$,
with the $\rho_i$ distinct representations of real type. Thus
$$
\rk \h \geq \sum_i \frac{\dim \rho_i}{2} - 1 = \frac{\dim \rho}{2} - r > \frac{\dim \rho}{4}
$$
as soon as $\dim \rho > 4r$. But this holds true if $\dim \rho_i \geq 5$, and this is a consequence of
lemma \ref{lemdimD} (2) because $n \geq 5$ and $\rho \not\in \Ind\LRef$.
Since $\dim \rho > 8$ we get the conclusion when $\rho$ has real type by lemma \ref{lemlie}. The remaining case
is when $\rho$ has complex type and $\dim \rho \leq 21$. But there are no $\rho \in \Irr''_0(A)$ of
complex type in type $D_6$, and $\dim \rho \geq 35$ when $n \geq 7$ by lemma \ref{lemdimD} (3) so we get the conclusion
by lemma \ref{lemlie2}.

 \subsection{$\rho \in \Irr''_0(\OO)$, case (1)}
 Here $\rho = (\la,\mu)$ with $\la \neq \mu$, $\la \neq \mu'$ and $(\la,\mu) \neq (\la',\mu')$. If the restriction is multiplicity free,
we have the conclusion by lemmas \ref{leminductreel} and \ref{lemdimD} (4) at least when $n \geq 6$.

We thus assume it is not. It is easily checked that
 it can happen only in the following situation :
  $\mu$ is deduced from $\la$ by removing one box at a given row, and $\la = \la'$ ;  then
  $\{ \la, \la \} = \{ \la \}^+ + \{ \la \}^-$, which restricts to $2((\la ))$ if in addition $(n-1)/2$ is odd.
  Since we assumed $n \geq 6$, note that this condition implies $n \geq 9$.

 The restriction can be written $2 \rho_0 + \rho_1 + \dots + \rho_r$,
 and
 $$
 \rk \h \geq \sum_i \frac{\dim \rho_i}{2} - 1 = \frac{\dim \rho - \dim \rho_0}{2}  -(r+1) > \frac{\dim \rho}{4}
 $$
iff $\dim \rho >  2 \dim \rho_0 + 4(r+1)$, that is $\dim \rho_1 + \dots + \dim \rho_r \geq m r > 4(r+1)$
where $m = \min \dim \rho_i$,
that is $m > 4(1+\frac{1}{r})$. By lemma \ref{lemdimD} we know $\dim \rho_i \geq 48$, so this condition
is fulfilled. Moreover $48 \leq \dim \rho < 4 \rk \h$ implies $\dim \rho <   (\rk \h + 1)^2$. We have multiplicities at most $2$, and $\g = \rho(\AA)$ does not contain
any $\sl_2$ by lemma \ref{lemnosl2}, hence $\g$ is a simple Lie algebra according to \cite{IH2}, lemma 3.3 (II). Thus $\rk \h > (\dim \rho)/4$
and $\dim \rho > 8$ implies $\g = \so(V_{\rho})$ by lemma \ref{lemlie}.

\subsection{$\rho \in \Irr''_0(\OO)$, case (2)}
Here $\rho = (\la, \mu)^{\pm}$ with $\la = \la'$, $\mu = \mu'$ and $\la \neq \mu$. Here the restriction is
multiplicity-free. If all the constituents in the restriction have real type, then we can conclude as in \ref{refsubsectcas36}.
The remaining situation is when there is $\alpha$ with $\alpha \nearrow \la$ and $\alpha = \mu$, that is $\mu \nearrow \la$.
In that case we have a multiplicity-free decomposition $\Res \rho = \{ \mu \}^{\pm +} + \{ \mu \}^{\pm - } + \rho_1 + \dots + \rho_r$ with the $\rho_i$ having real type.
Note that this implies $n-1 = 4m$. Since $n \geq 6$, this implies $n \geq 9$, and in particular all
the constituents of $\Res \rho$ have dimension at least $48$. If $\rho$ has real type,
then we are done by lemma \ref{leminductreel} ; if $\rho$ has complex type, by lemma \ref{lemlie2} we only
need to prove $\rk \h > (\dim \rho)/4$ and this is straightforward.

\subsection{Exceptional types}
\label{refsubsectexcept}
First notice that, once the character table of $W$ and $W_0$ is known, as well as the
induction table, the proof of the theorem for $W$ can be partly automatized : if
there are no irreducible representations of $W$ susceptible to lead to the the exceptional
types $\sl_4$ and $\so_8$ and if the induction table is multiplicity free, then one can check
systematically whether the conditions of lemma \ref{lemlie} (if we are in type $\so$) or \cite{IH2} lemma 3.2 (in type $\sl$)
are satisfied. In addition, in type $D_n$, representations factoring through $\mathfrak{S}_n$
can be considered tackled. This works for $D_5$, $D_7$, $D_{11}$,  and $E_7$.

When $W$ has type $D_6$, the only representations
remaining to check  are 
\begin{itemize}
\item
the restriction to $\OO$ of the $\{ [2,1] \}^{\pm}$, which are irreducible hence of real
type. They have dimension 40, their restriction to $\OO_0 = \Ker(W_0 \to \{ \pm 1 \})$ are multiplicity free and $\h$ has rank 10.
Since $10 > 40/5$ we get the conclusion by \cite{IH2}, lemma 3.5 (2).
\item the restriction to $\OO$ of $([1,1],[3,1])$, which is irreducible of real type. It has dimension 45, its restriction to $\OO_0$ is multiplicity free and $\h$ has rank 9. Since $9 = 45/5$, we cannot use the same argument. We compute the linear rank
of $\g = \rho(\AA)$ by computer, and get $\dim \g = 990  = \dim \so(V_{\rho})$.
\end{itemize}

When $W$ has type $E_6$,
the representations that we need to check are 

\begin{itemize}
\item the restriction to $\OO$ of the one labelled $\varphi_{15,16}$ in CHEVIE's convention. It has real type and dimension 15. The Lie algebra $\h$ has rang $2$.
We compute the linear rank of $\g = \rho(\mathcal{A})$ by computer : by lemma \ref{lemgensufort} we
know that it is generated by the images of the $(su)-(su)^{-1}$ for $s,u$ simple reflections with $su \neq us$,
and we have matrix models of this representation over $\Q$. By computer one readily gets
$\dim \rho(\AA) = 105 = \dim \so(V_{\rho})$ hence $\g = \so(V_{\rho})$.
\item the restriction to $\OO$ of the ones labelled $\varphi_{30,15}$ in CHEVIE's convention. It has real type and dimension 30.
The restriction is multiplicity-free, $\h$ has rank $7 > 30/5$, and we can use \cite{IH2} lemma 3.5 (2) to get the conclusion.

\item the restriction to $\OO$ of the ones labelled $\varphi_{20,20}$ in CHEVIE's convention. It has real type and dimension 20.
The restriction is multiplicity-free, $\h$ has rank 2. 
We need again to use a computer, and find $\dim \g = 190 = \dim \so_{20}$, as wanted.

\item the two constituents of the restriction to $\OO$ of the one labelled $\varphi_{90,8}$ in CHEVIE's convention.
They have complex type and dimension 45. The restriction is multiplicity-free, we have $\rk \h = 19$. Since $19 > 45/4$
we can use \cite{IH2} lemma 3.5 (2) to get the conclusion.

\end{itemize}

Finally, when $W$ has type $E_8$, the remaining ones all have real type.
They are the restrictions to $\OO$ of the representations of $W$ labelled
$\phi_{6075,22},  \phi_{3240,31}, \phi_{4536,23},\phi_{5600,21}$
of dimensions 6075, 3240,4536, 5600,
plus the two constituents $\phi_{7168,17}^+, \phi_{7168,17}^-$ (of dimension 3584) of the restriction of
the representation labelled $\phi_{7168,17}$. 
One easily gets in all these cases that $\rk \h > (\dim \rho)/4$, hence by \cite{IH2} lemma 3.4 it is sufficient
to prove that $\g$ is simple. Since the multiplicities are at most 2, it is easily checked to be a consequence of  \cite{IH2} lemma 3.3 (II),
provided we know that $\rho(\AA)$ has no simple ideal of type $\sl_2$. This has been proved in lemma \ref{lemnosl2} .

\end{document}